\documentclass[a4paper,11pt]{article}
\usepackage{geometry}
\usepackage{xspace}
\usepackage{graphicx}
\usepackage{epic}
\usepackage{eepic}

\usepackage{cpw}

\newcommand{\BB}{\ensuremath{\blackboardbold{B}}}
\newcommand{\graph}{\ensuremath{\mathop{\mathrm{graph}}}}
\renewcommand{\dim}{\ensuremath{\mathop{\mathrm{dim}}}}
\newcommand{\diam}{\ensuremath{\mathop{\mathrm{diam}}}}

\title{The stability index for dynamically defined Weierstrass functions}

\author{C.P.\ Walkden \ and T.\ Withers\footnote{A preliminary
    version of some of the results in this paper were contained in the
    second author's PhD thesis, financially supported by EPSRC and the
    School of Mathematics, University of Manchester.}}

\date{11th August 2017}

\begin{document}

\maketitle

\renewcommand{\thefootnote}{}

\begin{abstract}
Let $\hat{T} : X \times \RR \to X \times \RR$ given by $\hat{T}(x,t) =
(Tx, g_x(t))$ be a skew-product dynamical system where $T : X \to X$
is a mixing conformal expanding map and, for each $x \in X$, $g_x :
\RR \to \RR$ is an affine map of the form
$g_x(t)=-f(x)+\lambda(x)^{-1}t$.  Under a suitable contraction
hypotheses on $\lambda$ there exists a measurable function $u: X \to
\RR$ such that $\graph(u) = \{(x,u(x)) \mid x\in X\}$ is
$\hat{T}$-invariant and divides $X \times \RR$ into two regions,
$\BB^+$ and $\BB^-$, consisting of points that are repelled under
iteration by $\hat{T}$ to $\pm\infty$.  These two regions act as
basins of attraction to $\pm \infty$ in the sense of Milnor.  The two
basins have a complicated local structure: a neighbourhood of a point
$(x,t) \in \BB^+$ will typically intersect $\BB^-$ in a set of
positive measure.  The stability index (as introduced by Podvigina and
Ashwin \cite{podviginaashwin:11} for general Milnor attractors) is the
rate of polynomial decay of the measure of this intersection.  We
calculate the stability index at typical points in $X \times \RR$.  We
also perform a multifractal analysis of the level sets of the
stability index.
\end{abstract}

\section{Introduction and statement of results}
\label{sec:intro}

Given a dynamical system $T : X \to X$, a \dfn{Milnor attractor}
\cite{milnor:85} is a closed invariant set $A$ for which the basin of
attraction $\BB(A)$ (defined to be the set of points $x \in X$ for
which the omega-limit set $\omega(x) \subset \BB(A)$) has positive
measure and there is no strictly smaller closed set $A'\subset A$ for
which $\BB(A')=\BB(A)$.  We note that $\BB(A)$ is not required to be
an open set.  If $T$ has two, or more, attractors ($A_1, A_2$, say)
then the basins may have a complicated local structure: given a point
$x \in \BB(A_1)$, any neighbourhood of $x$ may intersect $\BB(A_2)$ in
a set of positive measure.  The notions of riddled basins and
intermingled basins were introduced, following numerical observations,
in \cite{alexander:92}.  A basin is \dfn{riddled} if its complement
intersects every ball in a set of positive measure; two basins are
\dfn{intermingled} if each ball that intersects one basin in a set of
positive measure also intersects the other basin in a set of positive
measure.  The study of riddled basins has attracted considerable
interest at the interface of dynamical systems and physics
\cite[for example]{greborgi:83, alexander:92, ott:93, sommerer:93, hiromichi:01,
  keller:17}.

In \cite{podviginaashwin:11}, the notion of \dfn{stability index} for
a basin was introduced (see also \cite{greborgi:83}).  The stability
index measures the extent to which a basin $\BB$ is riddled at a
given point.  Specifically, let $B_r(x) := \{ y \in X \mid d(x,y) <
r\}$.  For a given measure $\mu$, define the stability index
$\sigma(x)$ to be $\sigma(x) := \sigma^+(x)-\sigma^-(x)$ where
\begin{equation}
\label{eqn:stabindexgeneral}
 \sigma^+(x) = \lim_{r\to0} \frac{1}{\log r} \log \left( \frac{
   m(B_r(x))\cap \BB}{m(B_r(x))} \right),\ \ \ 
 \sigma^-(x) = \lim_{r\to0} \frac{1}{\log r} \log \left(1- \left( \frac{
   m(B_r(x))\cap \BB}{m(B_r(x))} \right) \right).
\end{equation}

Skew-products provide a class of particularly rich, yet tractable,
examples of dynamical system.  In \cite{mohdroslan:15}, skew-products
of the form $\hat{T} : [0,1]^2 \to [0,1]^2$, $\hat{T}(x,t) = (Tx,
h(x,t))$, where $T$ is a skewed doubling map and $h$ belongs to a
family of piecewise linear transformations are considered.  There are
two attractors: $\BB^- = X \times \{0\}$, $\BB^+ = X \times \{1\}$ and
the stability index at a.e.\ point of the form $(x,0)$ is calculated.

In \cite{keller:14}, skew-products of the form $\hat{T} : X \times
[0,1] \to X \times [0,1]$ of the form $\hat{T}(x,t) = (Tx, h(x,t))$
where $T$ is an invertible hyperbolic map, and $h(x,t)= h_1(x)h_2(t)$
with $h_2(t)$ a strictly increasing $C^{1+\alpha}$ concave function.
Again $X \times \{0\}$ and $X \times \{1\}$ are attractors and the
stability index at Lebesgue a.e.\ point of the form $(x,0)$ is
calculated.

Results similar to those above were put into the context of
thermodynamic formalism, in \cite{keller:17}.  Here, skew-products of
the form $\hat{T} : [0,1]^2 \to [0,1]^2$, $\hat{T}(x,t)=(Tx, g_x(t))$
are considered where the map $T$ is assumed to be a Markov expanding
map and, for each $x \in [0,1]$, $g_x : [0,1] \to [0,1]$ is a
diffeomorphism with negative Schwartzian derivative.  There are three
invariant graphs $u^- \leq u^c \leq u^+$ and the graphs of $u^-$ and
$u^+$ are attractors with riddled basins.  The stability index at
almost every (with respect to an appropriate measure) point $(x,t)$ is
calculated.
 
In \cite{mohdroslan:15, keller:14, keller:17}, the stability index
is typically found to be the ratio of two Lyapunov exponents
(corresponding to the exponential rate of contraction in the fibre
direction and the Lyapunov exponent of the base transformation)
multiplied by a constant related to the asymptotic distribution of the
invariant graph.

In this paper, we consider skew products over $C^{1+\alpha}$ conformal
expanding maps $T : X \to X$ and with fibre $\RR$ where the dynamics
in the fibre direction is affine.  Specifically, for $\alpha$-Holder
continuous functions $f : X \to \RR^+$ and $\lambda : X \to \RR^+$ we
define
\begin{equation}
\label{eqn:skewprod}
 \hat{T} : X \times \RR \to X \times \RR,\ \hat{T}(x,t) =  ( Tx, -f(x)+\lambda(x)^{-1}t) =: (Tx, g_x(t)).
\end{equation}
Note that, by replacing $T$ by $T^2$, we can assume without loss of
generality that $\hat{T}$ is orientation-preserving in each fibre.
Note that, for a fixed $t \in \RR$, the map $x \mapsto g_x(t)$ is
$\alpha$-\Holder.  We define $g_x^n(t)$ by $\hat{T}^n(x,t) =: (T^nx, g_x^n(t))$.

Under appropriate contraction hypotheses on $\lambda$ that ensure that
$\hat{T}$ expands in the fibre direction, skew products of the form
(\ref{eqn:skewprod}) possess an invariant graph, namely a function $u
: X \to \RR$ such that $\graph(u) = \{ (x,u(x)) \mid x \in X\}$ is
$\hat{T}$-invariant.  Under our hypotheses, $u$ will be measurable but
not continuous.  The graph of $u$ divides $X \times \RR$ into two
regions (up to a set of measure zero), one consisting of points that
are repelled to $+\infty$ in the fibre direction under iteration by
$\hat{T}$ and the other consisting of points repelled to $-\infty$.
We regard $X \times \{-\infty\}$ and $X \times \{\infty\}$ as
attractors for the skew-product $\hat{T}$.  We define
\[
 \BB^+ = \left\{(x,t) \in X\times \RR \mid \lim_{n\to\infty} g_x^n(t) = \infty \right\},\ \ \
 \BB^- = \left\{(x,t) \in X\times \RR \mid \lim_{n\to\infty} g_x^n(t) = -\infty \right\}
\]
and refer to these as the \dfn{basins of attraction to $\pm\infty$}, respectively.

We define the stability index of these basins as follows.  We let
$B_r(x,t) := B_r(x) \times [t-r,t+r] \subset X \times \RR$ be a
neighbourhood of $(x,t) \in X \times \RR$.  Let $\mu$ denote an
appropriate $T$-invariant probability measure on $X$ and let $m$
denote Lebesgue measure on $\RR$.  Define
\begin{equation}
\label{eqn:stabindex1}
 \Sigma_{\mu, r}^+(x,t) := \frac{ \mu \times m (B_r(x,t) \cap \BB^+)}{\mu \times m (B_r(x,t)},\ \ \ 
 \Sigma_{\mu, r}^-(x,t) := \frac{ \mu \times m (B_r(x,t) \cap \BB^-)}{\mu \times m (B_r(x,t)}
\end{equation}
and
\begin{equation}
\label{eqn:stabindex2}
 \sigma_\mu^+(x,t) = \lim_{r \to 0} 
 \frac{\log \Sigma_{\mu,r}^+(x,t)}{\log r},\ \ \ 
 \sigma_\mu^-(x,t) = \lim_{r \to 0} 
 \frac{\log \Sigma_{\mu,r}^-(x,t)}{\log r}. 
\end{equation}
We define the \dfn{stability index} $\sigma_\mu(x,t) :=
\sigma_\mu^+(x,t)-\sigma_\mu^-(x,t)$.  Note that
$\sigma_\mu^{\pm}(x,t) \geq 0$.  If $\Sigma^{\pm}_{\mu,r}(x,t) = 0$
for all $r < r_0$ (for some $r_0>0$) then we set
$\sigma_\mu^{\pm}(x)=\infty$.  If $\Sigma^{\pm}_{\mu,r}(x,t) = 1$ for
all $r < r_0$ then we set $\sigma_\mu^{\pm}(x,t) =0$.  We remark that
at most one of $\sigma_\mu^+(x,t), \sigma_\mu^-(x,t)$ can be non-zero
(see Lemma~\ref{lem:onepositiveimpliesotherzero}).

Typically, the invariant graph $u$ can be written in the form
\begin{equation}
\label{eqn:invgraphtemp}
 u(x) = \sum_{n=0}^{\infty} \lambda^{n+1}(x)f(T^nx)
\end{equation}
where $\lambda^{n+1}(x) := \lambda(x) \lambda(Tx) \cdots
\lambda(T^{n}x)$ and $\lambda^0(x):=1$.  As a particular example, take
$T(x)=bx \bmod 1$ (where $b \in \NN, b \geq 2$), $\lambda(x)=\lambda
\in (0,1)$, $\lambda b >1$, $f(x) = \cos 2\pi x$ then $u(x) =
\sum_{n=0}^{\infty} \lambda^{n+1} \cos 2\pi b^nx$, the classical
Weierstrass function.  (Note that, as $\lambda < 1$, this function is
continuous and so $\graph(u)$ divides $X \times \RR$ into two open
sets and neither basin is riddled with the other.)  For this reason,
we call functions of the form (\ref{eqn:invgraphtemp})
\dfn{dynamically defined Weierstrass functions}.

Assuming that $\lambda$ has a negative Lyapunov exponent for an
appropriate measure $\mu$, Stark \cite{stark:99} (cf.\ also
\cite{hadjiloucasnicolwalkden:02}) proved that an invariant graph $u$
exists and is given by (\ref{eqn:invgraphtemp}) $\mu$-a.e.  Generically,
$u$ is not continuous and is only measurable.  However $\graph(u)$
still divides (mod 0) $X \times \RR$ into two basins corresponding to attractors at $+\infty$ and $-\infty$.

The hypotheses we impose on the base dynamics and on the skew product
are as follows.  
\begin{itemize}
\item[(H1)] 
The base dynamical system $T : X \to X$ is a $C^{1+\alpha}$ conformal
expanding map or a uniformly expanding $C^{1+\alpha}$ Markov map of
the interval.  We assume that $T$ is topologically mixing.
\item[(H2)]
The function $f : X \to \RR$ is $\alpha$-\Holder continuous and $f > 0$.
\item[(H3)] 
The function $\lambda : X \to \RR$ is $\alpha$-\Holder continuous and
$\lambda>0$.  Moreover, there exists an equilibrium state $\mu$
corresponding to a \Holder continuous potential $\phi$ such that $\int
\log \lambda\,d\mu <0$ and a $T$-invariant probability measure $\zeta$
such that $\int \log \lambda\,d\zeta > 0$.  We assume without loss of
generality that $\phi$ is normalised so that the pressure $P(\phi)
=0$.  (Equilibrium states and pressure are defined in
\S\ref{subsec:thermodynamics}.)
\item[(H4)]
The skew product is partially hyperbolic: $\mathfrak{m}(\lambda) \mathfrak{m}(|T'|)^\alpha \geq \kappa^{-1} > 1$.  (Here $\mathfrak{m}(h) = \inf_{x \in X} h(x)$.)
\end{itemize}
The invariant measure $\zeta$ need not be an equilibrium state;
indeed, $\zeta$ could be a Dirac point mass at a fixed point for $T$.

We shall show that, under (H1)--(H4), an invariant graph $u$ exists
$\mu$-a.e., and we calculate the stability index at $\mu$-a.e.\ point.
We also calculate the multifractal spectrum of the stability index.

As a specific example that satisfies (H1)--(H4), take $T(x)=2x\bmod 1$
to be the doubling map.  Let
\[
 f(x) = \frac{2+\sin 2\pi x}{5},\ 
 \lambda(x) = \frac{4}{5} + \frac{\cos 2\pi x}{4}.
\]
Take $\mu$ to be Lebesgue measure and $\zeta$ to be the Dirac point
mass at $0$.  Then $\int \log \lambda\,d\mu < -0.24$ and $\int \log
\lambda\,d\zeta = \log 21/20 > 0$.  Hypothesis (H4) is satisfied as
$\mathfrak{m}(\lambda) \mathfrak{m}(|T'|)=1.1$.
Figure~\ref{fig:example} illustrates the structure of the invariant
graph $u$ and the basins.
\begin{figure}[h]
\begin{center}
\includegraphics[width=80mm]{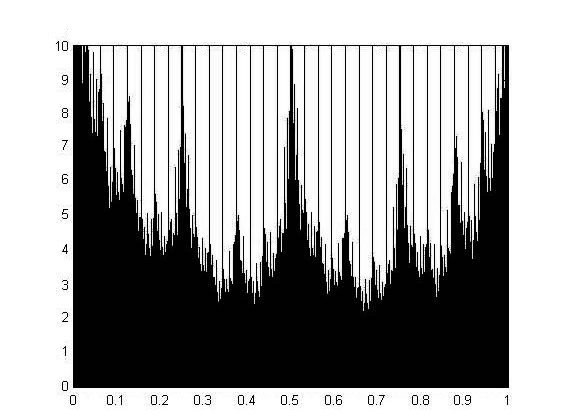}
\end{center}
\caption{The shaded region denotes the basin $\BB^-$.  Note that at
  the dyadic rationals (the pre-images of $0$ under $T$) the invariant
  graph is infinite.}
\label{fig:example}
\end{figure}

In the example illustrated in Figure~\ref{fig:example}, it appears
that the proportion of $X \times \RR$ occupied by $\BB^-$ decreases as
we move towards $+\infty$.  More specifically, for a fixed $t >0$,
consider the horizontal section $X_{(t)} := \{x \in X \mid (x,t) \in
\BB^-\}$.  We prove that under (H1)--(H4), $\mu(X_{(t)})$ decays
polynomially fast as $t\to\infty$.
\begin{theorem}
\label{thm:horizontal}
Assume that (H1)--(H3) hold and recall that $\phi$, the potential for
the equilibrium state $\phi$, is normalised so that the pressure
$P(\phi)=0$.  Then $\lim_{t\to\infty} -\log \mu(X_{(t)})/\log t = s^*$
where $s^*>0$ is the unique positive solution to the pressure equation
$P(\phi + s\log\lambda) = 0$.
\end{theorem}
We will also see that $s^*$ is related to the $L^p(\mu)$ class of the
invariant graph $u$ (Lemma~\ref{cor:loynesindex}).  We call the
constant $s^*$ the \dfn{Loynes exponent} (cf.~\cite{lelarge:07}).

%We define the Lyapunov exponent of $\lambda$ with respect to $\mu$ by
%\begin{equation}
%\label{eqn:lyapexplambda}
%\lim_{n\to\infty} \frac{1}{n} \log \lambda^n(x) = \int \log \lambda\,d\nu =: \chi(\n%u)
%\end{equation}
%for $\nu$-a.e. $x\in X$.

Our main result is the following.
\begin{theorem}
\label{thm:stabindex}
Assume that (H1)--(H4) hold.
\begin{itemize}
\item[(i)] 
The basin $\BB^+$ is riddled with $\BB^-$ and for $\mu$-a.e.\ $x \in
X$ and all $t > u(x)$ we have
\[
 \sigma_\mu(x,t) = -\sigma_\mu^{-}(x,t) = \frac{s^*\int \log \lambda\,d\mu}{\int \log |T'|\,d\mu}  < 0 
\]
where $s^*$ is as in Theorem~\ref{thm:horizontal}.
% and $\chi(\mu)$ is the Lyapunov exponent of $\lambda$ with respect to $\mu$.
\item[(ii)] The basin $\BB^-$ is not riddled with with
  $\BB^+$. Indeed, for $\mu$-a.e.\ $x\in X$ and all $t < u(x)$, there
  exists $r_0 >0$ such that for all $0 < r < r_0$
\[
 \mu \times m (B_r(x,t) \cap \BB^{-}) = \mu\times m (B_r(x,t))
 \]
so that $\sigma_\mu(x,t)=\infty$.
\item[(iii)]
For $\mu$-a.e.\ $x$, we have $\sigma_\mu(x,u(x))=0$.
\end{itemize}
\end{theorem}
In particular (and in contrast to the class of skew-products
considered in \cite{keller:17}), under (H1)--(H4), $\BB^+$ is not
intermingled with $\BB^-$.

Note that, in the case that $m$ denotes Lebesgue measure,
(\ref{eqn:stabindexgeneral}) is (in the limit as $r \to 0$) the
density of $\BB^-$ at $x$, by the Lebesgue Density Theorem.  Thus the
stability index is a form of local dimension, and this motivates many
of the arguments.  

We give a multifractal analysis of the Hausdorff dimension of the
level sets of the stability index.  Define
\[
 K_{\mu}(\sigma) = \{ x \in X \mid \sigma_\mu(x,t) =
 -\sigma\ \mbox{for all}\ t>u(x)\}
\]
to be the level sets of the stability index $\sigma_\mu(x,t)$.  We are
interested in the Hausdorff dimension of $K_{\mu}(\sigma)$.  We first
state a special case when $X=[0,1]$ and $\mu$ is the SRB measure.
\begin{proposition}
\label{prop:multifractalmarkov}
Assume (H1)--(H4) and assume that $T$ is a mixing uniformly expanding
Markov map of the interval.  Let $\mu$ be the SRB measure for $T$.
Define $S(q)$ by $P(-S(q)\log |T'| + qs^*\log \lambda)=0$ (here $s^*$
is as in Theorem~\ref{thm:horizontal}).  Let $\mu_q$ denote the
equilibrium state with potential $-S(q)\log |T'| + qs^*\log \lambda$
and let $\sigma(q) = -s^*\int \log \lambda\,d\mu_q/\int \log
|T'|\,d\mu_q$.  Then $\mu_0 = \mu$, $S(0)=1, S(1)=1$ and $S(q)$ is a
real analytic, strictly convex function.
\begin{itemize}
\item[(i)]
We have that 
\[
 {\dim}_H \left\{ x \in X \mid \sigma_\mu(x,t)=
 \frac{s^*\int\log\lambda\,d\mu}{\int \log |T'|\,d\mu}\ \mbox{for
   all}\ t > u(x) \right\} =1.
\]
\item[(ii)]
There exists a unique $q^* \in (0,1)$ such that $\int \log \lambda\,d\mu_{q^*}=0$.  
\item[(iii)] 
The functions $\sigma \mapsto \dim_H K_{\mu}(\sigma)$ and $q \mapsto
S(q)$ form a Legendre transform pair.  In particular, $\dim_H
K_{\mu}(\sigma(q)) = S(q)-q\sigma(q)$ for $q \in (-\infty,q^*)$.
\end{itemize}
\end{proposition}
More generally we have
\begin{theorem}
\label{thm:multifractal}
Assume (H1)--(H4).  Define $S(q)$ by $P(-S(q)\log |T'| + qs^*\log
\lambda)=0$ (here $s^*$ is as in Theorem~\ref{thm:horizontal}), let
$\mu_q$ denote the equilibrium state with potential $-S(q)\log |T'| +
qs^*\log \lambda$ and let $\sigma(q) = -s^*\int \log
\lambda\,d\mu_q/\int \log |T'|\,d\mu_q$.  Then $S(0)=\dim_H X$ and
$S(q)$ is a real analytic, strictly convex function. 
\begin{itemize}
\item[(i)]
There exists a unique $q^* \in \RR$ such that $\int \log
\lambda\,d\mu_{q^*} = 0$.  
\item[(ii)] 
The functions $\sigma \mapsto \dim_H K_{\mu}(\sigma)$ and $q \mapsto
S(q)$ form a Legendre transform pair.  In particular, $\dim_H
K_{\mu}(\sigma(q)) = S(q)-q\sigma(q)$ for $q \in (-\infty,q^*)$.
\end{itemize}
\end{theorem}
A similar result to Proposition~\ref{prop:multifractalmarkov} was
obtained in \cite{hiromichi:01}.  Here, a skew-product acting on
$[0,1] \times [0,2]$ similar to that in \cite{mohdroslan:15} (with $T$
a skewed tent map equipped with Lebesgue measure and a piecewise
linear skewing function) was considered. The set $[0,1] \times \{0\}$
is a Milnor attractor, and the multifractal structure of the stability
index for points of the form $(x,0)$ is calculated and is of the form
in Figure~\ref{fig:multifrac}.

\section{Preliminary definitions and results}
\label{sec:background}

\subsection{Conformal expanding maps}
\label{subsec:Markov}

Let $M$ be a smooth Riemannian manifold.  Let $T : M \to M$ be
$C^{1+\alpha}$ and conformal.  Suppose that $X \subset M$ is a compact
$T$-invariant set such that
\begin{itemize}
\item[(i)]
there exists $C>0$ and $\theta \in (0,1)$ such that $\|d_xT^n(v)\|
\geq C\theta^{-n}\|v\|$ for all $x \in X$, all $v \in T_xM$ and all $n
\geq 0$;
\item[(ii)]
there exists an open neighbourhood $U$ of $X$ such that $X = \{x \in U \mid T^nx \in U\ \mbox{for all}\ n \geq 0\}$.
\end{itemize}
We consider $T: X \to X$ and we assume (without loss of generality)
that $T$ is mixing.  As $T$ is conformal, we write $d_xT(x) =
T'(x)O(x)$ where $O(x)$ is orthogonal and $T' : X \to \RR$.  We also
use an adapted Riemannian metric and take, without loss of generality,
$C=1$ in (i).  Note that there exists $r_0>0$ such that, for any $x \in X$,  $T$
restricted to $B_{4r_0}(x)$ has well-defined inverses.
 
We can also consider the case when $X=[0,1]$ and $T : X \to X$ is a
topologically mixing uniformly expanding Markov map.  In this case,
there is a partition $\{ t_i \mid 0 = t_0 < t_1 < \cdots <
t_{n-1}<t_n=1\}$ such that, with $I_j=(t_j, t_{j+1})$, we have that if
$T(I_i)\cap I_j \not=\emptyset$ then $T(I_i) \supset I_j$; (ii) for
each $j$, $T|_{I_j}$ is a $C^{1+\alpha}$ diffeomorphism onto its image
and $|T'(x)| \geq \theta^{-1} >1$ for some $\theta \in (0,1)$.  Note
that if $x \not= t_j$ then there exists $r_0 = r_0(x)>0$ such that $T$
restricted to $B_{4r_0}(x)$ has well-defined inverses.

A cover $\mathcal{R} = \{R_1, \ldots, R_k\}$ of $X$ by closed subsets
of is said to be a \dfn{Markov partition} if (i) each $R_j$ is the
closure of its interior, (ii) $\mathop{\mathrm{int}} R_i \cap
\mathop{\mathrm{int}} R_j = \emptyset$ for $i\not= j$, and (iii) for
each $j$, $T(R_j)$ is the union of sets in $\mathcal{R}$.  It is
well-known that $T$ possesses Markov partitions with arbitrarily small
diameters.  When we calculate the stability index at a point $(x,t)$,
we choose $r_0$ as above and then choose a Markov partition
$\mathcal{R}$ to have diameter smaller than $r_0$.

We write $S_nh(x) := \sum_{j=0}^{n-1}h(T^jx)$.  The following estimate
is well-known.
\begin{lemma}
\label{lem:boundeddistortion}
Let $r_0$ be such that $T$ restricted to $B_{r_0}(z)$ has well-defined
inverse branches.  Let $n >0$ and let $\tau$ be any branch of
$T^{-n}$.  Then for all $x,y \in \tau (B_{r_0}(z))$ we have $d(T^jx,
T^jy) \leq \theta^{n-j}$ for $0\leq j \leq n-1$.  Moreover, let $h: X
\to \RR$ be \Holder continuous.  Then there exists $C_h>0$ such that,
whenever $x,y \in \tau (B_{r_0}(z))$ then $|S_nh(x)-S_nh(y)| \leq
C_h$.
\end{lemma}
In particular, applying
Lemma~\ref{lem:boundeddistortion} to $\log \lambda$ and $\log |T'|$ (and writing
$C_\lambda$, $C_T$ in place of $C_{\log \lambda}$, $C_{\log |T'|}$, respectively), we have that for all $n$
and all $x,y \in \tau(B_{r_0}(z))$
\begin{equation}
\label{eqn:bdddistlambda}
 C_\lambda^{-1} \leq \frac{\lambda^n(x)}{\lambda^n(y)} \leq C_\lambda,\ \ \
 C_T^{-1} \leq \frac{ \prod_{j=0}^{n-1} |T'(T^jx)|^{-1}}{ \prod_{j=0}^{n-1} |T'(T^jy)|^{-1}} \leq C_T.
\end{equation}

We let
\[
 [i_0, \ldots, i_{n}] := \{ x \in X \mid T^jx \in
 I_{i_j}\ \mbox{for}\ j=0,1,\ldots, n\}
\]
and define this to be a cylinder of rank $n$.  If $x \not\in
\bigcup_{n=0}^{\infty} T^{-n}\bdry\mathcal{R}$ then we write
$A_n(x)$ to be the unique cylinder of rank $n$ that contains $x$.

We recap the definition of a Moran cover \cite{pesinweiss:97}.  Given
$r>0$ define $n_r(x)$ to be the unique integer such that
\begin{equation}
\label{eqn:morancover}
 \prod_{j=0}^{n_r(x)} | T'(T^jx)|^{-1} 
 < r 
 \leq
 \prod_{j=0}^{n_r(x)-1} | T'(T^jx)|^{-1}.
\end{equation}
Fix $x$ and consider the cylinder set $A_{n_r(x)}(x)$.  Then $x \in
A_{n_r(x)}(x)$.  If $y \in A_{n_r(x)}(x)$ and $n_r(y) \leq n_r(x)$
then $A_{n_r(x)}(x) \subset A_{n_r(y)}(x)$.  Let $A_{(r)}(x)$ denote
the largest (in diameter) cylinder such that $x \in A_{(r)}(x)$ and
$A_{(r)}(x) = A_{n_r(y)}(x)$ for some $y \in A_{n_r(x)}(x)$ and
$A_{n_r(z)}(x) \subset A_{(r)}(x)$ for all $z \in A_{(r)}(x)$.  The
sets $A_{(r)}(x)$ (as $x$ varies) either coincide or are disjoint
except at their endpoints.  They form a partition $\mathcal{U}_r$ of
$X$ which we call a \dfn{Moran cover}.  We enumerate the sets in
$\mathcal{U}_r$ as $\{ A_{n_r(x_i)}(x_i) \mid 1 \leq i \leq \ell_r \}$.
We note that it follows from (\ref{eqn:morancover}) that
\begin{equation}
\label{eqn:moran2}
 r\|T'\|_\infty^{-1} \leq \prod_{j=0}^{n_r(x)} | T'(T^jx)|^{-1} 
\end{equation}
and that $r\|T'\|_\infty^{-1} \leq \diam A_{n_r(x_i)}(x_i) < r$ for $1
\leq i \leq \ell_r$.

A Moran cover forms the most efficient cover of $X$ by cylinders of
diameter no greater than $r$.  An important property of Moran covers
is the following.  There exists $M>0$ such that,
for any $x \in X$ and $r>0$ sufficiently small, the number of sets in
$\mathcal{U}_r$ that have non-empty intersection with $B_r(x)$ is
bounded above by $M$.  We call $M$ the \dfn{Moran multiplicity
  factor}.

In general, $\sup \{ n_r(x_i)-n_r(x_j) \mid 1 \leq i,j \leq \ell_r \}$
tends to infinity as $r\to0$.  Below, we will only consider elements
of $\mathcal{U}_r$ that cover a ball $B_r(x)$.  With this restriction,
we have the following lemma.
\begin{lemma}
\label{lem:moranranksbdd}
Fix $x \in X$ and choose a Markov partition as above.  Choose the
elements of $\mathcal{U}_r$ that have non-empty intersection with
$B_r(x)$, with labelling chosen so that
\[
 B_r(x) \subset \bigcup_{i=1}^{M} A_{n_r(x_i)}(x_i),
\]
where $M$ is the Moran multiplicity factor.  Moreover, there exists
$L$, independent of $r$, such that $|n_r(x_i)-n_r(x_k)| \leq L$ for
all $1 \leq i,k \leq M$.
\end{lemma}
\begin{proof}
The existence of the Moran multiplicity factor is derived in
\cite[\S20]{pesin:97}.

Choose $k$ such that $n_r(x_k) \leq n_r(x_i)$ for $1 \leq i \leq
M$.  By (\ref{eqn:bdddistlambda}), (\ref{eqn:morancover}) and
(\ref{eqn:moran2}) we have
\[
 (C_T \|T'\|_\infty)^{-1}
 \leq
 \frac{ \prod_{j=0}^{n_r(x_i)} |T'(T^jx_i)|^{-1}}
      { \prod_{j=0}^{n_r(x_k)} |T'(T^jx_k)|^{-1}}
 \times
 \frac{ \prod_{j=0}^{n_r(x_k)} |T'(T^jx_k)|^{-1}}
      { \prod_{j=0}^{n_r(x_k)} |T'(T^jx_i)|^{-1}}
 =
 \prod_{j=n_r(x_k)+1}^{n_r(x_i)}  |T'(T^jx_i)|^{-1}
 \leq
 C_T \|T'\|_\infty.
\]
Hence $(C_T\|T'\|_\infty)^{-1} \leq
\|T'\|_\infty^{n_r(x_k)-n_r(x_i)}$.  It follows that
$n_r(x_i)-n_r(x_k) \leq \log( C_T \|T'\|_\infty)^{-1}/\log
\|T'\|_\infty$. 
\end{proof}

\subsection{Thermodynamic formalism}
\label{subsec:thermodynamics}

Let $\phi : X \to \RR$ be continuous.  The \dfn{pressure} of $\phi$, $P(\phi)$, is defined to be
\begin{equation}
\label{eqn:varprinciple}
 P(\phi) := \sup \left\{ h_{\nu}(T) + \int \phi\,d\nu \mid \nu\ \mbox{is a}\ T\mbox{-invariant probability measure} \right\}
\end{equation}
where $h_{\nu}(T)$ is the entropy of $T$ with respect to $\nu$.

If $\phi$ is \Holder continuous then there exists a unique measure
$\mu=\mu_\phi$, the \dfn{equilibrium state with potential $\phi$},
that achieves this supremum.  The measure $\mu$ satisfies the
\dfn{Gibbs property}, namely that there exists $C_\mu>0$ such that for
all $x \in X$ and all $n \in \NN$
\begin{equation}
\label{eqn:gibbs}
\frac{1}{C_\mu} \leq \frac{ \mu(A_n(x))}{ \exp \left( S_n\phi(x) - nP(\phi) \right)} \leq C_\mu.
\end{equation}
By replacing $\phi$ by $\phi-P(\phi)$ there is no loss in assuming
that $P(\phi)=0$ (it is clear that $\mu_{\phi-P(\phi)}=\mu_\phi$); we
say that $\phi$ is \dfn{normalised} if $P(\phi)=0$.

We need the following distortion bound for measures that satisfy the
Gibbs property.  Note that for any cylinder $A_n(x)$ of rank $n$, the
restriction $T|_{A_n(x)} : A_n(x) \to X$ is a bijection.
\begin{lemma}
\label{lem:kellerbdddist}
There exists $D>1$ such that the following property holds.  Suppose $B
\subset B_r(x)$ is a Borel subset and $r<r_0$.  Let $A_{n_r(x_j)}(x_j)$, $1 \leq j
\leq M$, be a Moran cover for $B_r(x)$ and suppose that the indexing
is chosen so that $A_{n_r(x_1)}(x_1) \subset B_r(x)$.  Then
\begin{equation}
\label{eqn:kellerbdddist}
 D^{-1}    
 \mu(T^{n_r(x_1)}B)
 \leq
 \frac{\mu(B)}{\mu \left( \bigcup_{j=1}^{M} A_{n_r(x_j)}(x_j) \right)}
 \leq
 D \mu(T^{n_r(x_1)}B).
\end{equation}
%for any cylinder $A_n=A_n(x)$ of rank $n$
%and any Borel subset $B \subset A_n$, we have
%\begin{equation}
%\label{eqn:kellerbdddist}
% D^{-1} \mu(T^n B) \leq \frac{\mu(B)}{\mu(A_n)} \leq D \mu(T^n B).
%\end{equation}
\end{lemma}
\begin{proof}
As cylinders in a Moran cover overlap only on their boundaries (which
have zero $\mu$-measure), we have that $\mu(\bigcup_{j=1}^{M}
A_{n_r(x_j)}(x_j)) = \sum_{j=1}^{M} \mu (A_{n_r(x_j)}(x_j))$.  We also
note that as $r<r_0$ then $\bigcup_{j=1}^M A_{n_r(x_j)}(x_j) \subset
B_{4r_0}(x)$ and we can apply Lemma~\ref{lem:boundeddistortion}.

Let $1 \leq i,j \leq M$ and suppose that $n_r(x_i) < n_r(x_j)$.  By
Lemma~\ref{lem:moranranksbdd} we can write $n_r(x_j)=n_r(x_i)+k$ with
$k \leq L$, where $L$ is independent of $r$.  As $\phi$ is normalised
we have that $\phi(x) < 0$; indeed, $-\|\phi\|_\infty \leq \phi(y)
\leq -\mathfrak{m}(\phi)$ for all $y\in X$.  By (\ref{eqn:gibbs}) we
have
\begin{equation}
\label{eqn:kellerbdd1}
 C_\mu^{-1}e^{-L\|\phi\|_\infty}e^{S_{n_r(x_i)}\phi(x_j)}
 \leq
 C_\mu^{-1}e^{S_{n_r(x_j)}\phi(x_j)} 
 \leq 
 \mu(A_{n_r(x_j)}(x_j))
 \leq
 C_\mu e^{S_{n_r(x_j)}\phi(x_j)}
 \leq 
 C_\mu e^{S_{n_r(x_i)}\phi(x_j)}.
\end{equation}
We also note from Lemma~\ref{lem:boundeddistortion} that there exists
a constant $C_{\phi}>0$ such that
\begin{equation}
\label{eqn:kellerbdd2}
 C_\phi^{-1} \leq \frac{e^{S_{n_r(x_i)}\phi(x_j)}}{e^{S_{n_r(x_i)}\phi(x)}} \leq C_\phi.
\end{equation}
From (\ref{eqn:kellerbdd1}), (\ref{eqn:kellerbdd2}) it follows that
there exists a constant $C_{\phi,\mu}$ independent of $r$ such that
for all $1 \leq j \leq M$ we have
\[
 C_{\phi,\mu}^{-1} e^{S_{n_r(x_1)}\phi(x)} 
 \leq
 \mu( A_{n_r(x_j)}(x_j) )
 \leq
 C_{\phi,\mu} e^{S_{n_r(x_1)}\phi(x)}.
\]
Hence 
\[
 MC_{\phi,\mu}^{-1} e^{S_{n_r(x_1)}\phi(x)} 
 \leq 
 \mu\left( \bigcup_{j=1}^{M}A_{n_r(x_j)}(x_j) \right)
 \leq
 MC_{\phi,\mu} e^{S_{n_r(x_1)}\phi(x)}.
\]

We first observe that (\ref{eqn:kellerbdddist}) holds when $I\subset
B_r(x)$ is a cylinder of sufficiently large rank.  To see this, let
$I=A_{n_r(x_1)+p}(y) \subset B$.  By (\ref{eqn:gibbs}) we have that
$\mu(I) \leq C_\mu \exp S_{n_r(x_1)+p}\phi(y) = C_\mu\exp S_{n_r(x_1)}\phi(y) \exp
S_p\phi(T^{n_r(x_1)}y)$.  Now
$T^{n_r(x_1)}I$ is a cylinder of rank $p$ containing $T^n(y)$, hence
$C_\mu^{-1}e^{S_p\phi(T^{n_r(x_1)}y)} \leq \mu(T^{n_r(x_1)}I) \leq C_\mu
e^{S_p\phi(T^{n_r(x_1)}y)}$.  By Lemma~\ref{lem:boundeddistortion} and
(\ref{eqn:gibbs}), for an appropriate constant $C_\phi> 0$, we have
that
\begin{eqnarray*}
 \mu(I) & \leq & C_\mu e^{S_{n_r(x_1)}\phi(y)}e^{S_{p}\phi(T^{n_r(x_1)}y)}\\
 &\leq&
 C_\mu^2 e^{S_{n_r(x_1)}\phi(y)} \mu(T^{n_r(x_1)}I) \\
 &\leq&
 C_\mu^2 C_\phi e^{S_{n_r(x_1)}\phi(x)} \mu(T^{n_r(x_1)}I) \\
 &\leq& 
 M^{-1} C_{\phi,\mu} C_\mu^3 C_\phi \mu\left(
 \bigcup_{j=1}^{M}A_{n_r(x_j)}(x_j) \right) \mu(T^{n_r(x_1)}I).
\end{eqnarray*}
The lower bound follows similarly.

Now let $I,J \subset B_r(x)$ be disjoint cylinders of rank at least
$n_r(x_1)$.  Then $T^{n_r(x_1)}I$ and $T^{n_r(x_1)}J$ are disjoint
cylinders.  It is straightforward to check that
(\ref{eqn:kellerbdddist}) holds when $B = I \cup J$.

Now let $B \subset B_r(x)$ be a Borel subset.  Let $\veps > 0$.
Choose a finite union of cylinders $C$ of rank at least $n_r(x_1)$
such that $\mu(C \Delta T^{n_r(x_1)}B) < \veps$.  Let $C' \subset
B_r(x)$ be such that $T^{n_r(x_1)}C'=C$; note that $C'$ is a finite
union of cylinders of rank at least $n_r(x_1)$ and that $\mu(C' \Delta B) <
\veps$.  Then
\begin{eqnarray*}
 \frac{\mu(B)}{\mu\left( \bigcup_{j=1}^M A_{n_r(x_j)}(x_j) \right)} 
 & \leq &
 \frac{\mu(C')}{\mu\left( \bigcup_{j=1}^M A_{n_r(x_j)}(x_j) \right)} +
 \frac{\veps}{\mu\left( \bigcup_{j=1}^M A_{n_r(x_j)}(x_j) \right)} \\
 &\leq &
 D \mu(T^{n_r(x_1)}C) + 
  \frac{\veps}{\mu\left( \bigcup_{j=1}^M
     A_{n_r(x_j)}(x_j) \right)} \\
 & \leq &
 D \mu(T^{n_r(x_1)}B) + \left(D+ \frac{1}{\mu\left( \bigcup_{j=1}^M
   A_{n_r(x_j)}(x_j) \right)}\right)\veps.
\end{eqnarray*}
As $\veps>0$ is arbitrary, the right-hand side of
(\ref{eqn:kellerbdddist}) holds for $B$.  The left-hand side follows
similarly.
\end{proof}

Given $s > 0$ we define the transfer operator $L_s$ by
\begin{equation}
\label{eqn:transferop}
 L_sw(x) = \sum_{y: Ty=x} e^{\phi(y)}e^{s\log \lambda(y)} w(y).
\end{equation}
Define $p(s) := P(\phi + s\log \lambda)$.  It is well-known that $L_s$
has spectral radius $e^{p(s)}$.  As $\phi$ is normalised we have
$p(0)=0$.  After possibly adding a coboundary to $\phi$, we can assume
that $L_s\mathbf{1} = \mathbf{1}$, where $\mathbf{1}$ denotes the
constant function.  Note that, as $\phi$ is normalised, we must have
that $\phi(x) < 0$ for all $x \in X$.

It is well-known that there is a Banach space $B$ of functions, which
contain the constants, such that $L_s : B \to B$ has $e^{p(s)}$ as a
simple maximal eigenvalue and the remainder of the spectrum is
contained inside a disc of radius $\gamma_s < e^{p(s)}$.  In
particular, we can write $L_s^n = e^{np(s)}\pi_s + O(\gamma_s^n)$
where $\pi_s$ is a projection operator.

\subsection{Stability index}
\label{subsec:stabindextech}

It is clear from (\ref{eqn:stabindex2}) that $\sigma_\mu^{\pm}(x,t)
\geq 0$, including the possibility that it is infinite.  We make the
following remark.
\begin{lemma}
\label{lem:onepositiveimpliesotherzero}
Suppose that $\sigma_\mu^\pm(x,t)$ exists and $\sigma_\mu^\pm(x,t)
> 0$ or is infinite.  Then $\sigma_\mu^{\mp}(x,t)$ exists and
$\sigma_\mu^{\mp}(x,t)=0$.
\end{lemma}
\begin{proof}
Note that $\log r <0$ if $0<r<1$.  Suppose first that $\sigma_\mu^+(x,t) =
\delta > 0$.  Then, provided $r>0$ is sufficiently small, we have
that $\Sigma_{\mu,r}^+(x,t) > \veps^{\delta/2}$.  As $\mu\times
m(\graph(u))= 0$, we have that $\Sigma_{\mu,r}^-(x,t) = 1-
\Sigma_{\mu,r}^+(x,t) > 1-r^{\delta/2}$.  Hence
\[
 \frac{1}{\log r} \log \Sigma_{\mu,r}^-(x,t) < \frac{\log (1-r^{\delta/2})}{\log r}.
\]
Letting $r \to 0$, the claim follows.  A similar argument holds when $\sigma_\mu^\pm(x,t)= \infty$.
\end{proof}

\section{Structure of the invariant graph}
\label{sec:structure}

We define (formally) $u(x)$ by
\begin{equation}
\label{eqn:invgraph}
 u(x) = \sum_{n=0}^{\infty} \lambda^{n+1}(x)f(T^nx).
\end{equation}
Let $X_u$ denote the set of points $x \in X$ for which there exists $C(x)>0$ $\eta>0$ and $N(x)\in \NN$ such that $\lambda^n(x) < C(x) e^{-\eta n}$ for all $n \geq N(x)$.  We will only consider $u$ to be defined on $X_u$ (but see
Proposition~\ref{prop:uexists}).

We define $g_x^n(t)$ by iterating (\ref{eqn:skewprod}), specifically $\hat{T}^n(x,t) = (T^nx, g_x^n(t))$.  It is straightforward to see that
\begin{equation}
\label{eqn:defgxn}
 g_x^n(t) = -\sum_{j=0}^{n-1} \lambda^{n-1-j}(x)^{-1}f(T^jx) + \lambda^n(x)^{-1}t 
  = \lambda^n(x)^{-1} \left( -\sum_{j=0}^{n-1} \lambda^{j+1}(x)f(T^jx) + t \right).
\end{equation}
We introduce the notation
\begin{equation}
\label{eqn:snflambda}
 S_{n,\lambda}f(x) := \sum_{j=0}^{n-1} \lambda^{j+1}(x)f(T^jx)
 \end{equation}
so that $\hat{T}^n(x,t) = \left(T^nx, \lambda^{n}(x)^{-1}\left( -S_{n,\lambda}f(x) +t \right) \right)$.
\begin{proposition}
\label{prop:uexists}
Let $\nu$ be an ergodic invariant measure such that $\int \log
\lambda\,d\nu < 0$.  Then $\nu(X_u)=1$.  Moreover, $u$ is
$\nu$-measurable, $\graph(u) = \{(x,u(x)) \mid x\in X_u\}$ is
$\hat{T}$-invariant, and $u$ is unique in the sense that if $v$ is a
$\nu$-measurable function with a $\hat{T}$-invariant graph then $v=u$
$\nu$-a.e.
\end{proposition}
\begin{proof}
That $\nu(X_u) =1$ follows immediately from Birkhoff's Ergodic
Theorem.  

Let $u_n(x) = \sum_{j=0}^{n-1} \lambda^{j+1}(x)f(T^jx)$.  If $x \in
X_u$ then there exists $\eta>0$ such that for all sufficiently large
$n$ we have $\lambda^n(x) < C(x) e^{-\eta n}$.  Then
$|S_{n+1,\lambda}f(x)-S_{n,\lambda}f(x)| = \lambda^{n+1}(x)f(T^nx)
\leq \|f\|_\infty C(x) e^{-\eta(n+1)}$ and it follows that $u_n$ is
Cauchy, and so converges.

That $X_u$ is $T$-invariant and $\graph(u)$ is $\hat{T}$-invariant are straightforward calculations.  

To prove uniqueness, suppose that $v$ is $\nu$-measurable and has a $\hat{T}$-invariant graph.  Then $v(x)-u(x) = \lambda^n(x) (v(T^nx)-u(T^nx))$.   As $v-u$ is measurable, there exists a constant $C_1>0$ and a set $V$ of positive $\nu$-measure such that $(v-u)(x) < C_1$ for all $x \in V$.  By ergodicity, for $\nu$-a.e.\ $x\in V$ there is a subsequence such that $T^{n_j}x \in V$  As $\lambda^{n_j}(x) \to 0$ $\nu$-a.e.\ it follows that $u(x)=v(x)$ $\nu$-a.e.
\end{proof}

Under the hypotheses of Proposition~\ref{prop:uexists}, it follows
from \cite{hadjiloucasnicolwalkden:02} that $u$ is continuous if and
only if there exists a continuous function $r$ such that
$f(x)=r(Tx)-\lambda(x)^{-1}r(x)$ and that generically this does not
happen.  We shall see below in Corollary~\ref{cor:loynesindex} that,
under hypotheses (H1)--(H3), the function $u$ is never continuous.

We now prove that the graph of $u$ determines the boundary between the two basins.
\begin{proposition}
\label{prop:graphdefinesbasins}
Suppose $x \in X_u$ so that $u(x)$ exists.  Then $(x,t) \in \BB^+$ if
and only if $t > u(x)$ and $(x,t) \in \BB^-$ if and only if $t <
u(x)$.
\end{proposition}
\begin{proof}
%Recall that $f, \lambda > 0$ so that $u(x)>0$.  Suppose $t > u(x)$ and write $t = u(x)+\delta_x(t)$.  Then
%\[
% g^n_x(t) = \lambda^n(x)^{-1} \left( -S_{n,\lambda}f(x) + u(x) + \delta_x(t) \right)
% = \lambda^n(x)^{-1} \left( \sum_{j=n+1}^\infty \lambda^{j+1}(x)f(T^jx) + \delta_x(t) \right) 
% \geq \lambda^n(x)^{-1}\delta_x(t)
%\]
%which tends to $\infty$ as $n \to \infty$.

Recall that $f, \lambda > 0$ so that $u(x)>0$.  Suppose $t < u(x)$ and
write $t = u(x) -\delta_x(t)$.  Provided that $n$ is sufficiently
large we have $0 < \sum_{j=n+1}^\infty \lambda^{j+1}(x)f(T^jx) <
\delta_x(t)/2$.  Hence
\begin{eqnarray*}
 g^n_x(t) & =&  \lambda^n(x)^{-1} \left( -S_{n,\lambda}f(x) + u(x) - \delta_x(t) \right) \\
 & =& \lambda^n(x)^{-1} \left( \sum_{j=n+1}^\infty \lambda^{j+1}(x)f(T^jx) - \delta_x(t) \right) 
 \leq -\frac{\delta_x(t)}{2} \lambda^n(x)^{-1}
\end{eqnarray*}
so that $(x,t) \in \BB^-$, noting that $\lambda^n(x)^{-1} \to \infty$
as $n\to\infty$ by the definition of $X_u$.  The argument for $t >
u(x)$ is analogous.
\end{proof}

\section{A thermodynamic Loynes exponent}
\label{sec:loynesindex}

For $s \geq 0$ recall that $p(s)= P(\phi+s\log \lambda)$ where $P$ denotes
the topological pressure.  It is well-known that $p(s)$ is a convex
analytic function of $s$.  
\begin{lemma}
\label{lem:existenceofregindex}
Assume (H1)--(H3).  Then there exists a unique $s^* > 0$ such that $p(s^*)=0$.  Moreover, $p'(s^*) > 0$ and $p'(s)$ is strictly increasing on an open interval $(\underline{s}, \overline{s})$ that contains $s^*$.
\end{lemma}
\begin{proof}
Recall from \cite{ruelle:78} that if $\phi$ is \Holder continuous,
$P(\phi)=0$ and has equilibrium state $\mu$ and $\psi$ is \Holder
continuous then $\partial P(\phi+t \psi)/\partial t|_{t=0} = \int
\psi\,d\mu$.  Moreover $\partial^2 P(\phi+t \psi)/\partial t^2|_{t=0}
\geq 0$ with equality if and only if $\psi$ is cohomologous to a
constant.

First note that $p(0)=0$ as $\phi$ is normalised.  By the above we
have that $p'(0)=\int \log \lambda\,d\mu < 0$.  As $\int \log
\lambda\,d\zeta > 0$, we see that $\log \lambda$ cannot be
cohomologous to a constant.  Hence $p(s)$ is strictly convex.

By the variational principle (\ref{eqn:varprinciple}).
$p(s) = \sup \{ h_{\nu}(T) + \int \phi\,d\nu + s\int \log \lambda\,d\nu \}$
where the supremum is taken over all $T$-invariant probability
measures $\nu$.  Hence $p(s) \geq h_{\zeta}(T) + \int \phi\,d\zeta +
s\int \log \lambda\,d\zeta$.  It follows that $p(s) \to \infty$ as
$s\to\infty$ as $\int \log \lambda\,d\zeta > 0$.

As  $p(s)$ is analytic and convex, it follows that there is a unique $s^* > 0$ such that $p(s^*)=0$.  Moreover, $p'(s^*) > 0$.   Hence there is an interval $(\underline{s}, \overline{s})$ containing $s^*$ on which $p'(s) > 0$.   As $p$ is convex, $p'$ is non-decreasing.  To see that $p'$ is strictly increasing on $(\underline{s}, \overline{s})$, suppose for a contradiction that $p''(s)=0$ on a subinterval of $(\underline{s}, \overline{s})$.   Then $p''(s)=0$ for all $s$, by analytic continuation, implying that $p'(s)$ is constant for all $s$; this contradicts $p'(0) < 0, p'(s^*)>0$.
\end{proof}

The goal of this section is to prove the following result.
\begin{proposition}
\label{prop:loynesindex}
Assume that (H1)--(H3) hold.  Let $u$ be the $\mu$-a.e.\ defined invariant graph for $\hat{T}$.  Let $s^*>0$ be the unique positive solution to $p(s)=0$.  Then
\begin{equation}
\label{eqn:loynesindex}
 \lim_{M\to\infty} \frac{-\log \mu \left( \left\{ x \in X \mid u(x)>M\right\} \right)}{\log M} = s^*.
\end{equation}
\end{proposition}

Before proving Proposition~\ref{prop:loynesindex}, we relate the
constant $s^*$ to the regularity of the invariant graph and also prove
Theorem~\ref{thm:horizontal}.
\begin{corollary}
\label{cor:loynesindex}
Assume that (H1)--(H3) hold.  Let $u$ be the
$\mu$-a.e.\ defined invariant graph for $\hat{T}$.   Then $u \in \mathcal{L}^{p}(\mu)$ if $p < s^*$ and $u \not\in \mathcal{L}^{p}(\mu)$ if $p > s^*$
\end{corollary}
\begin{proof}
Recall that $u>0$ $\mu$-a.e.  Let $U_n := \{ x \in X \mid u(x)^p > n\} = \{ x \in X \mid u(x) > n^{1/p}\}$.  Note that 
\[
 \bigcup_{n=0}^{\infty} U_{n+1} \times [n,n+1] \subset \{ (x,t) \in X \times \RR \mid 0 \leq t \leq u(x) \} \subset \bigcup_{n=0}^{\infty} U_n \times [n,n+1].
\]
Let $\veps > 0$.  Provided $n$ is sufficiently large (\ref{eqn:loynesindex}) implies that $n^{-(s^*+\veps)/p} \leq \mu(U_n) \leq n^{-(s^*-\veps)/p}$.  Hence $\int u\,d\mu = \mu \times m \{ (x,t) \in X \times \RR \mid 0 \leq t \leq u(x) \} < \infty$ if $\sum_{n=0}^{\infty} n^{-(s^*-\veps)/p} < \infty$.  Hence $u \in \mathcal{L}^p(\mu)$ if $p < s^*$.  Similarly, $u \not\in \mathcal{L}^p(\mu)$ if $p > s^*$.
\end{proof}

Hence if $s^*<1$ then $u$ will not be integrable; however, $u$ is always log-integrable.
\begin{corollary}
\label{cor:logintegrable}
Assume that (H1)--(H3) hold.  Then $\log^+ u := \max \{0, \log u\} \in L^1(\mu)$.
\end{corollary}
\begin{proof}
Let $V_n := \{x \in X \mid \log u(x)>n\} = \{ x\in X \mid u(x) >e^n\}$.  By Proposition~\ref{prop:loynesindex}, $\mu(V_n) < e^{-ns^*/2}$ provided $n$ is sufficiently large.  Note that
\[
 \graph (\log^+u) \subset \bigcup_{n=0}^{\infty} V_n \times [n,n+1]
\]
and that $\log^+u$ is positive.  Noting that $\sum_{n=0}^{\infty} \mu(V_n) < \infty$ as $\sum_{n=0}^{\infty} e^{-ns^*/2}<\infty$, the claim follows.
\end{proof}
\begin{proofof}{Theorem~\ref{thm:horizontal}}
By Proposition~\ref{prop:graphdefinesbasins} we can write $X_{(t)} = \{ x \in X \mid u(x)>t\}$ and the result follows immediately from Proposition~\ref{prop:loynesindex}.
\end{proofof}

We now prove Proposition~\ref{prop:loynesindex}; the arguments follow
those in \cite{keller:14, keller:17}.  We establish the limsup and
liminf in (\ref{eqn:loynesindex}) separately.
\begin{lemma}
\label{lem:upperbound1}
Let $s\in (0,s^*)$.  Then there exists $\delta_0 = \delta_0(s)> 0$ such that if $\delta \in (0,\delta_0)$ then there exists a constant $C=C(\delta,s) > 0$ with the following property: for all $M > 0$ and all $n\in \NN$ we have
\[
 \mu \left( \left\{ x \in X \mid M^{-1} \lambda^n(x) \geq e^{-2n\delta} \right\} \right) 
 \leq
  CM^{-s} e^{-ns\delta}.
\]
\end{lemma}
\begin{proof}
As $s< s^*$, choose $\delta>0$ such that $p(s)+4s\delta < 0$.  Note
that, by the spectral radius theorem, $\int L_s^n \mathbf{1}\,d\mu
\leq C e^{n(p(s)+4s\delta)}$ for some constant $C>0$.  Then
\begin{eqnarray*}
 \mu(\{ x \in X \mid M^{-1}\lambda^n(x) \geq e^{-2\delta n}\})
 & = &
 \mu(\{ x \in X \mid M^{-s} e^{s S_n\log\lambda(x)}e^{2s\delta n} \geq 1\}) \\
 & \leq &
 M^{-s} e^{2s\delta n} \int e^{s S_n\log \lambda(x)}\,d\mu \\
 & = & 
 M^{-s} e^{2s\delta n} \int L_s^n\mathbf{1}\,d\mu \\
 & = &
 C M^{-s} e^{-s\delta n}.
\end{eqnarray*}
\end{proof}

\begin{lemma}
\label{lem:upperbound2}
Assume that (H1)--(H3) hold.  Let $u$ be the $\mu$-a.e.\ defined invariant graph for $\hat{T}$.  Then
\begin{equation}
\label{eqn:upperbound2}
 \liminf_{M\to\infty} 
 \frac{-\log \mu \left( \left\{ x\in X \mid u(x) > M \right\} \right)}
 {\log M}
 \geq s^*.
\end{equation}
\end{lemma}
\begin{proof}
Let $M>1$.  As $u(x) = \sum_{n=0}^{\infty} \lambda^{n+1}(x)f(T^nx)$\ $\mu$-a.e.\ we have that
\[
 \mu \left( \left\{ x \in X \mid u(x) > M \right\} \right) \leq 
 \mu \left( \left\{ x \in X \mid \sum_{n=0}^{\infty} \lambda^{n+1}(x) > M\|f\|_{\infty}^{-1} \right\} \right) =: \mu(\Delta).
\]
Note that $M\|f\|_{\infty}^{-1} = M\|f\|_{\infty}^{-1} (1-e^{-\delta})\sum_{n=0}^{\infty} e^{-\delta n}$.  Let $\hat{M} = M\|f\|_{\infty}^{-1}(1-e^{-\delta})$.  Hence
\[
 \Delta = \left\{ x \in X \mid \hat{M}^{-1} \sum_{n=0}^{\infty} \lambda^{n+1}(x) > \sum_{n=0}^{\infty}e^{-\delta n} \right\}.
\]
If $x \in \Delta$ then there must exist $n\geq0$ such that $\hat{M}^{-1} \lambda^{n+1}(x) > e^{-\delta n}$.  From this observation and Lemma~\ref{lem:upperbound1}, we have
\[
 \mu(\Delta) \leq \sum_{n=0}^{\infty} \mu \left( \left\{ x \in X \mid \hat{M}^{-1} \lambda^{n+1}(x) > e^{-\delta n} \right\} \right)  
  \leq 
 \sum_{n=0}^{\infty} \hat{M}^{-s}Ce^{-(\delta/2)ns}
  \leq  C' M^{-s} 
\]
on summing the geometric series, for some constant $C'=C'(s,\delta)>0$.   Hence $\mu( \{ x \in X \mid u(x)> M\}) \leq C' M^{-s}$.  Taking logs, dividing by $-\log M$ and taking the liminf as $M\to\infty$ gives that the left-hand side of (\ref{eqn:upperbound2}) is at least $s$.  As this is true for any $s<s^*$, the result follows.
\end{proof}

We now prove the limsup in (\ref{eqn:loynesindex}).  This makes use of the fact that $f>0$.  The following large deviations theorem due to Plachky and Steinebach \cite{plachkysteinebach:75} is true far more generally and we state it in the setting that we shall use it.
\begin{theorem}[\cite{plachkysteinebach:75}]
\label{thm:plachky}
Let $(\underline{s},\overline{s})$ be an open interval containing $s^*$ and suppose that, for $s \in (\underline{s},\overline{s})$, $p(s)$ is a differentiable function with $p'(s)$ strictly monotone.  Suppose that
\begin{itemize}
\item[(i)]
$\int e^{s\log \lambda^n(x)}\,d\mu < \infty$ for all $s \in [0,\overline{s})$,
\item[(ii)]
we have
\begin{equation}
\label{eqn:plachky}
 \lim_{n\to\infty} \frac{1}{n} \int e^{s\lambda^n(x)}\,d\mu = p(s)
\end{equation}
for all $s \in (\underline{s},\overline{s})$.
\end{itemize}
Then
\[
 \lim_{n\to\infty} \frac{1}{n} \log \mu( \{ x \in X \mid \log \lambda^n(x) > np'(s^*) \} ) = p(s^*) - s^* p'(s^*).
\]
\end{theorem}
We check that our setting does indeed satisfy the hypotheses of Theorem~\ref{thm:plachky}.  As $p(s)$ is convex and not linear, $p'(s)$ is strictly increasing.  Hypothesis (i) of Theorem~\ref{thm:plachky} holds trivially as $\lambda$ is continuous, hence bounded.  We need only check the convergence in (\ref{eqn:plachky}).  To see this, simply note that
\[
 \lim_{n\to\infty} \frac{1}{n} \log \int e^{s\lambda^n(x)}\,d\mu
 = 
 \lim_{n\to\infty} \frac{1}{n} \log \int L_s^n \mathbf{1}\,d\mu
 = 
  \lim_{n\to\infty} \frac{1}{n} \log \int e^{np(s)}\pi_s\mathbf{1}+O(\gamma_s^n)\,d\mu
 = p(s).
\]

We can now apply Theorem~\ref{thm:plachky} to complete the proof of Proposition~\ref{prop:loynesindex}.
\begin{lemma}
\label{lem:loewrbound}
Assume that (H1)--(H3) hold.  Then
\[
 \limsup_{M\to\infty} \frac{-\log \mu(\{ x\in X \mid u(x) > M\})}{\log M} \leq s^*.
\]
\end{lemma}
\begin{proof}
Recall that both $f, \lambda$ are positive.  First note that, for any
$n$ and any $0 \leq m \leq n-1$, we have $u(x) \geq S_{n,\lambda}f(x)
\geq \lambda^m(x)\mathfrak{m}(f)$.

Let $\alpha = p'(s^*)>0$.  For each $M$, choose $m$ such that $\alpha(m-1) + \log \mathfrak{m}(f) \leq \log M < \alpha m + \mathfrak{m}(f)$.  Then
\begin{eqnarray*}
 \mu(\{ x \in X \mid u(x) > M \}) & \geq &
 \mu(\{ x \in X \mid \lambda^m(x)\mathfrak{m}(f) > M \}) \\
 & \geq &
 \mu(\{ x \in X \mid \log \lambda^m(x) > \alpha m\}).
\end{eqnarray*}
Hence
\begin{eqnarray*}
\frac{-\log \mu(\{ x \in X \mid u(x) > M\})}{\log M} 
& \leq &
\frac{-\log \mu(\{ x \in X \mid \log \lambda^m(x) > \alpha m\})}{\log \alpha(m-1)+\log \mathfrak{m}(f)} \\
& = &
\frac{-1}{\alpha} 
\frac{\log \mu(\{ x \in X \mid \log \lambda^m(x) > \alpha m\})}{m}
\frac{\alpha m}{\alpha(m-1)+\mathfrak{m}(f)}
\end{eqnarray*}
and the lemma follows from Theorem~\ref{thm:plachky} by letting $M$,
equivalently $m$, tend to $\infty$.
\end{proof}

\section{Stability index}
\label{sec:stabindex}

We are now in a position to prove Theorem~\ref{thm:stabindex}. 

\subsection{The upper basin}
\label{subsec:upper}

To calculate the stability index for points in $\BB^+$ we first prove
that the exponential fibre-wise growth rate for a.e.\ point above the
graph is given by the Lyapunov exponent of $\lambda$.  
%Recall that for an invariant measure $\mu$ we define $\chi(\mu) := \int \log \lambda\,d\mu$.
\begin{lemma}
\label{lem:upperbasinlyap}
Assume that (H1)--(H3) hold.   For $\mu$-a.e.\ $x \in X$ and for all $t > u(x)$ we have
\begin{equation}
\label{eqn:upperbasinlyap}
 \lim_{n\to\infty} \frac{1}{n} \log  g_x^n(t) = -\int \log \lambda\,d\mu.
\end{equation}
\end{lemma}
\begin{proof}
For $\mu$-a.e.\ $x\in X$, $u(x)$ is given by (\ref{eqn:invgraph}).
For such an $x$, let $\delta_x(t) = t-u(x)>0$.  Choose $N$ such that for all $n \geq N$ we have
\[
 0 \leq \sum_{j=n}^{\infty} \lambda^{j+1}(x)f(T^jx) < \delta_x(t).
\]
Hence
\begin{eqnarray*}
 g_x^n(t) & = & \lambda^n(x)^{-1}(-S_{n,\lambda}f(x)+t) \\
  & = &
 \lambda^n(x)^{-1} \left( \sum_{j=n}^{\infty}\lambda^{j+1}(x)f(T^jx) + \delta_x(t) \right)
\end{eqnarray*}
so that $\lambda^n(x)^{-1}\delta_x(t) \leq g_x^n(t) \leq 2
\lambda^n(x)^{-1}\delta_x(t)$.  Taking logarithms, dividing by $n$ and
letting $n\to\infty$ then gives (\ref{eqn:upperbasinlyap}).
\end{proof}

We require the following lemma.
\begin{lemma}
\label{lem:uppertechlemma}
Let $x \in X_u$.  For $t > u(x)$ for which $\sigma_\mu^-(x,t)$ exists we have
\begin{eqnarray*}
\lefteqn{ \liminf_{r\to 0} \frac{1}{\log r}
 \log \left( \frac{ \mu( \{ y \in B_r(x) \mid u(y) > t- r\}) }{\mu(B_r(x))} \right) \leq 
 \sigma_\mu^{-}(x,t) } && \\
 & \leq & 
 \limsup_{r\to 0} \frac{1}{\log r}
 \log \left( \frac{ \mu( \{ y \in B_r(x) \mid u(y) > t+ r\}) }{\mu(B_r(x))} \right).
\end{eqnarray*}
\end{lemma}
\begin{proof}
First note that
\begin{eqnarray*}
 \lefteqn{\{ y \in B_r(x) \mid u(y) > t+r\} \times [t-r, t+r] }&&\\
 & \subset &
 B_r(x,t) \cap \BB^-\\
 &\subset &
 \{ y \in B_r(x) \mid u(y) > t-r\} \times [t-r, t+r]. 
\end{eqnarray*}
Noting that $\mu \times m(B_r(x,t)) = \mu(B_r(x)) \times 2r$ and taking logs we obtain
\[
 \log \left(\frac{\mu ( \{ y \in A_n(x) \mid u(y) > t+r \})}{\mu(B_r(x))}\right) 
 \leq 
 \log \Sigma_{\mu,r}^{-}(x,t)
 \leq 
 \log \left( \frac{\mu ( \{ y \in B_r(x) \mid u(y) > t-r \})}{\mu(B_r(x))} \right).
\]
Dividing by $\log r$ (noting that $\log r < 0$) then gives the result.
\end{proof}

%\begin{lemma}
%\label{lem:uppertechlemma}
%Let $x \in X_u$ and let $\mu(A_n(x)) = \veps_n$.  For $t > u(x)$ for which $\sigma_\mu^-(x,t)$ exists we have
%\begin{eqnarray*}
%\lefteqn{ \liminf_{n\to\infty} \frac{1}{\log \veps_n}
% \log \left( \frac{ \mu( \{ y \in A_n(x) \mid u(y) > t- \veps_n/2\}) }{\mu(A_n(x))} \right) \leq 
% \sigma_\mu^{-}(x,t) } && \\
% & \leq & 
% \limsup_{n\to\infty} \frac{1}{\log \veps_n}
% \log \left( \frac{ \mu( \{ y \in A_n(x) \mid u(y) > t+ \veps_n/2\}) }{\mu(A_n(x))} \right).
%\end{eqnarray*}
%\end{lemma}
%\begin{proof}
%First note that
%\begin{eqnarray*}
% \lefteqn{\{ y \in A_n(x) \mid u(y) > t+\veps_n/2\} \times [t-\veps_n/2, t+\veps_n/2] }&&\\
% & \subset &
% U_n \cap \BB^-\\
% &\subset &
% \{ y \in A_n(x) \mid u(y) > t-\veps_n/2\} \times [t-\veps_n/2, t+\veps_n/2]. 
%\end{eqnarray*}
%Noting that $\mu \times m(U_n) = \mu(A_n(x)) \times \veps_n$ and taking logs we obtain
%\[
% \log \frac{\mu ( \{ y \in A_n(x) \mid u(y) > t+\veps_n/2 \})}{\mu(A_n(x))} 
% \leq 
% \log \Sigma_{\mu,\veps_n}^{-}(x,t)
% \leq 
% \log \frac{\mu ( \{ y \in A_n(x) \mid u(y) > t-\veps_n/2 \})}{\mu(A_n(x))}.
%\]
%Dividing by $\log \veps_n$ (noting that $\log \veps_n < 0$) then gives the result.
%\end{proof}

The following bounded distortion estimate allows us to move between different points in $A_n(x)$.  Note that it is here that we require the partial hyperbolicity assumption (H4).
\begin{lemma}
\label{lem:upperboundeddist}
Assume that (H1)--(H4) hold.  Let $x,y \in A_n(x)$.  Then there exists
$C_{f,\lambda}> 0$, independent of $x,y,n$, such that
\begin{equation}
\label{eqn:upperboundeddist}
% \left| \sum_{j=0}^{n-1} \lambda^{j+1}(x) f(T^jx) - \sum_{j=0}^{n-1} \lambda^{j+1}(y) f(T^jy) \right| \leq C_{f,\lambda} \lambda^n(x).
 \left| S_{n,\lambda}f(x) - S_{n,\lambda}f(y) \right| \leq C \lambda^n(x).
\end{equation}
\end{lemma}

\begin{proof}
We write $|h|_\theta := \sup_{x,y} |h(x)-h(y)|/d(x,y)^\alpha$ for the \Holder semi-norm of $h$.  

Note that by Lemma~\ref{lem:boundeddistortion} we have
\begin{eqnarray*}
 \left| \lambda^{j+1}(x)f(T^jx) - \lambda^{j+1}(y) f(T^jy) \right| 
  &\leq &
 \lambda^{j+1}(x) \left| f(T^jx) - f(T^jy) \right| +
 |f(T^jy)| \left| \lambda^{j+1}(x)- \lambda^{j+1}(y) \right|\\
 & \leq & 
 |f|_\alpha \theta^{\alpha(n-j)} \lambda^{j+1}(x) + \|f\|_{\infty}  \left| \lambda^{j+1}(x)- \lambda^{j+1}(y) \right|.
\end{eqnarray*}

We can bound
\begin{equation}
\label{eqn:boundeddistlambda1}
 \left| \lambda^{j+1}(x)- \lambda^{j+1}(y) \right|
 \leq 
 \sum_{i=0}^{j} \lambda^{j-i}(T^{i+1}y) | \lambda(T^ix)-\lambda(T^iy)| \lambda^i(x).
\end{equation}
As $\log \lambda$ is $\alpha$-\Holder, we have
\begin{equation}
\label{eqn:boundeddistlambda}
 \left| \log \lambda^{j-i}(T^{i+1}x) - \log \lambda^{j-i}(T^{i+1}y) \right| 
 \leq 
 \sum_{k=0}^{j-i-1} |\log \lambda|_\alpha d(T^{k+i+1}x,T^{k+i+1}y)^\alpha 
 \leq 
 \sum_{k=0}^{\infty} |\log \lambda|_\alpha \theta^{k\alpha} = D
\end{equation}
where $D>0$ is independent of $x,y,i,j,n$.  Hence
\[
 e^{-D} \leq  \frac{ \lambda^{j-i}(T^{i+1}x)}{ \lambda^{j-i}(T^{i+1}y)} \leq e^{D}.
\]

As $\lambda$ is $\alpha$-\Holder continuous, we have $| \lambda(T^ix)-\lambda(T^iy)| \leq |\lambda|_\alpha \theta^{\alpha(n-i)}$.
We can then bound
\[
 \left| \lambda^{j+1}(x)- \lambda^{j+1}(y) \right| \leq e^{D}
 |\lambda|_\alpha \mathfrak{m}(\lambda) \lambda^{j+1}(x) \sum_{i=0}^{j}
 \theta^{\alpha(n-i)}.
\]
Noting that $\sum_{i=0}^{j} \theta^{\alpha(n-i)} \leq
(\theta^{-\alpha}-1)^{-1} \theta^{\alpha(n-j)}$ we have
\[ 
 \left| \lambda^{j+1}(x)- \lambda^{j+1}(y) \right|
 \leq
 C \lambda^{j+1}(x) \theta^{\alpha(n-j)}
\]
for some constant $C>0$ independent of $x,y,j,n$.

Hence
\[ 
 \left| \lambda^{j+1}(x)f(T^jx) - \lambda^{j+1}(y) f(T^jy) \right|
 \leq
 \sum_{j=0}^{n-1} C' \theta^{\alpha(n-j)}\lambda^{j+1}(x)
\]
for some constant $C'$ independent of $x,y,n$.  

Recall that by (H4) we have
$\mathfrak{m}(\lambda)^{-1} \theta^{\alpha} = \kappa <1$.  Hence
\begin{eqnarray*}
 \sum_{j=0}^{n-1} C' \theta^{\alpha(n-j)} \lambda^{j+1}(x) 
 & 
 = &
 C' \lambda^{n}(x) \sum_{j=0}^{n-1} \theta^{\alpha(n-j)} \lambda^{n-j-1}(T^{j+1}x)^{-1} \\
 & \leq &
 C' \lambda^{n}(x) \sum_{j=0}^{n-1} \theta^{\alpha(n-j)} \mathfrak{m}(\lambda)^{-(n-j-1)} \\
 & \leq &
 C' \mathfrak{m}(\lambda) \lambda^{n}(x) \sum_{j=0}^{n-1} \kappa^{n-j}
 \leq C_{f,\lambda} \lambda^n(x)
\end{eqnarray*}
for some constant $C_{f,\lambda}$, summing the geometric progression.
\end{proof}

We can now obtain the following bounded distortion estimate.
\begin{lemma}
\label{lem:boundeddistortiong}
Assume that (H1)--(H4) hold.  Suppose that $x \in X_u$ and $t > u(x)$.  Then there exists $K \geq 1$, depending on $x,t$, such that for all sufficiently large $n$ and all $y \in A_n(x)$ we have
\[
 K^{-1} \leq \frac{ g_x^n(t)}{g_y^n(t)} \leq K.
\]
\end{lemma}
\begin{proof}
First note that as $u(x)>0$ and $g^n_x(\cdot)$ is orientation preserving, we have $g^n_x(t) > g^n_x(u(x)) = u(T^nx) > 0$.

Recall from (\ref{eqn:defgxn}) that
\[
  \frac{ g_x^n(t)}{g_y^n(t)}  
 =
   \frac{ \lambda^n(x)^{-1} \left( -S_{n,\lambda}f(x) + t \right)}{ \lambda^n(y)^{-1} \left( -S_{n,\lambda}f(y) + t \right)}.
\]

By (\ref{eqn:bdddistlambda}), $C_{\lambda}^{-1} \leq
\lambda^n(x)^{-1}/\lambda^n(y)^{-1} \leq C_\lambda$.

Let $t-u(x) = \delta_x(t)>0$.  Then
\[
 -S_{n,\lambda}f(x) + t =
 -S_{n,\lambda}f(x) + u(x) +\delta_x(t) = 
  \sum_{j=n}^{\infty}  \lambda^{j+1}f(T^jx) + \delta_x(t) > \delta_x(t).
\]
Provided $n$ is sufficiently large we have that $\sum_{j=n}^{\infty}  \lambda^{j+1}f(T^jx) < \delta_x(t)$.  Hence $-S_{n,\lambda}f(x)+t \leq 2\delta_x(t)$.

By Lemma~\ref{lem:upperboundeddist}, we have
\begin{eqnarray*}
 -S_{n,\lambda}f(y) + t 
 & = &
 -S_{n,\lambda}f(y) + S_{n,\lambda}f(x)  -S_{n,\lambda}f(x) + u(x) + \delta_x(t)\\
 & \geq & 
  -C_{f,\lambda} \lambda^n(x) + \sum_{j=n}^{\infty} \lambda^{j+1}(x)f(T^jx) + \delta_x(t)\\
 & \geq &
 -C_{f,\lambda} \lambda^n(x) + \delta_x(t).
\end{eqnarray*}
As $x \in X_u$, we have $\lambda^n(x) \to 0$ as $n\to\infty$.  Hence
provided $n$ is sufficiently large then
$-S_{n,\lambda}f(y) + t \geq \delta_x(t)/2$.  Similarly,
\[
 -S_{n,\lambda}f(y)+t \leq C_{f,\lambda}\lambda^n(x)
 +\sum_{j=n}^\infty \lambda^{j+1}(x)f(T^jx) + \delta_x(t).
\]
By choosing $n$ sufficiently large we can assume that
$C_{f,\lambda}\lambda^n(x)<\delta_x(t)$ and $\sum_{j=n}^\infty
\lambda^{j+1}(y)f(T^jy) \leq \delta_x(t)$.  Hence
$-S_{n,\lambda}f(y)+t < 3\delta_x(t)$.

Hence
\[
 \frac{1}{3} \leq \frac{-S_{n,\lambda}f(x) + t}{-S_{n,\lambda}f(y) + t} \leq 4.
\]
This suffices to prove the lemma.
\end{proof}

We can now calculate the stability index for points above the graph.
\begin{lemma}
\label{lem:boundstabindexupper}
Assume that (H1)--(H4) hold.  For $\mu$-a.e.\ $x \in X$ and all $t > u(x)$ we have
\[
 \sigma_\mu^-(x,t) = \frac{- s^*\int \log \lambda\,d\mu}{\int \log |T'|\,d\mu}.
\]
\end{lemma}
\begin{proof}
We first prove that $\sigma_\mu^-(x,t) \leq -s^* \int \log \lambda\,d\mu/\int \log |T'|\,d\mu$
for $\mu$-a.e.\ $x\in X$.

Given $x \in X_u$, choose $r_0$ as in \S\ref{subsec:Markov} and choose
a Markov partition with diameter no more than $r_0$.  We assume that
$r<r_0$.

For each $r$, let $A_{n_r(x_j)}(x_j)$, $1 \leq j \leq M$ be a Moran
cover of $B_r(x)$.  Let $N = \min \{ n_r(x_j) \mid 1 \leq j \leq M
\}$.  As $\diam A_{n_r(x_j)}(x_j) < r$ for each $j$ and $\diam B_r(x)
=2r$, without loss of generality we can choose the indexing so that
$A_{n_r(x_1)}(x_1) \subset B_r(x) \subset \bigcup_{j=1}^{k}
A_{n_r(x_j)}(x_j) \subset B_{4r}(x)$.  Note that $x \in
A_{n_r(x_1)}(x_1)$.

Let $t^+ > t$.  Then $t^+ > t+r \in \BB^+$ provided that $r$ is
sufficiently small.  Note that
\[
 \mu( \{ y \in X \mid u(y) > t^+\}) 
 \leq
 \mu( \{ y \in X \mid u(y) > t+r \}).
\]
We have that
\[
 \frac{
 \mu( \{ y \in B_r(x) \mid u(y) > t^+\})}
 { \mu (B_r(x))} 
 \geq
 \frac{
 \mu( \{ y \in A_{n_r(x_1)}(x_1) \mid u(y) > t^+\})}
 { \mu \left( \bigcup_{j=1}^{M} A_{n_r(x_j)}(x_j) \right) }.
\]
By Lemma~\ref{lem:kellerbdddist} we see that
\[
 \frac{
 \mu( \{ y \in A_{n_r(x_1)}(x_1) \mid u(y) > t^+\})}
 { \mu \left( \bigcup_{j=1}^{M} A_{n_r(x_j)}(x_j) \right) }
 \geq 
 D^{-1}
 \mu( T^{n_r(x_1)}\{ y \in A_{n_r(x_1)}(x_1) \mid u(y) > t^+\}).
\]

We claim that
\[
 T^{n_r(x_1)}\{ y \in A_{n_r(x_1)}(x_1) \mid u(y) > t^+\} \supset
 \{ z \in A_{n_r(x_1)-N}(T^{n_r(x_1)}x_1) \mid u(z) > K g^{n_r(x_1)}_{x_1}(t^+)\}.
\]
To see this, let $z \in X$ and suppose $u(z) > K
g^{n_r(x_1)}_{x_1}(t^+)$.  There exists a unique $y \in
A_{n_r(x_1)}(x_1)$ such that $T^{n_r(x_1)}y=z$.  Note that $x_1,y$ are
in the same cylinder of rank $n_r(x_1)$; hence by
Lemma~\ref{lem:boundeddistortiong}, we have that $u(T^{n_r(x_1)}y) =
u(z) > Kg^{n_r(x_1)}_{x_1}(t^+) > g^{n_r(x_1)}_y(t^+)$.  As
$g^{n_r(x_1)}_y(\cdot)$ is orientation preserving, we have that $u(y) > t^+$.

Hence
\[
  \frac{
 \mu( \{ y \in A_{n_r(x_1)}(x_1) \mid u(y) > t^+\})}
 { \mu \left( \bigcup_{j=1}^{M} A_{n_r(x_j)}(x_j) \right) }
  \geq 
 D^{-1} 
 \mu(\{ z \in X \mid u(z) > K g^{n_r(x_1)}_{x_1}(t^+)\}).
\]
The above, together with Lemma~\ref{lem:boundeddistortiong}, gives
that
\[
\frac{ \mu( \{ y \in B_r(x) \mid u(y) > t^+ \})}{\mu(B_r(x))}
 \geq
 D^{-1} \mu(\{ z \in X \mid u(z) > K^2 g^{n_r(x_1)}_{x}(t^+)\}).
\]
For convenience, write $n_r := n_r(x_1)$ and note that $n_r \to
\infty$ as $r \to 0$.  Dividing the above by $\log r$ (again, noting
that $\log r < 0$) it follows from Lemma~\ref{lem:uppertechlemma} that
\begin{eqnarray}
 \sigma_\mu^-(x,t) 
  & \leq & 
 \limsup_{r\to 0} \frac{1}{\log r}
 \log \left( \frac{ \mu(\{ y \in B_r(x) \mid u(y) > t+r \})}{\mu(B_r(x))} \right) \nonumber \\
 & \leq &
 \limsup_{r \to 0} \frac{1}{\log r}
 \log D^{-1} \mu(\{ z \in X \mid u(z) > K^2g^{n_r}_x(t^+)\}).
 \label{eqn:boundstabindexupper}
\end{eqnarray}

We split the right-hand side of (\ref{eqn:boundstabindexupper}) as
\[
 \frac{n_r}{-\log r}
 \times
 \frac{\log Kg^{n_r}_x(t^+)}{n_r}
 \times
 \frac{-\log D^{-1} \mu \left( \left\{ z \in X \mid u(z) > K^2 g_x^{n_r}(t^+)\right\} \right)}{\log K^2g_x^{n_r}(t^+)}
\]

It follows from (\ref{eqn:morancover}) and Birkhoff's Ergodic Theorem
that for $\mu$-a.e.\ $x$
\[
 \lim_{r\to 0} \frac{-\log r}{n_r} = \lim_{n\to\infty} \frac{-1}{n} S_n\log |T'(T^jx)| = -\int \log |T'|\,d\mu 
\]

That $n_r^{-1}\log K^2g^{n_r}_x(t^+) \to -\int \log \lambda\,d\mu$ as $n_r\to\infty$
for $\mu$-a.e.\ $x \in X$ follows from Lemma~\ref{lem:upperbasinlyap}

As $t^+ > u(x)$, Proposition~\ref{prop:graphdefinesbasins} implies that $g^{n_r}_x(t^+) \to \infty$ as $r \to 0$.  By Proposition~\ref{prop:loynesindex} we have that
\[
  \frac{-\log D^{-1} \mu \left( \left\{ z \in X \mid u(z) > 
 K^2 g_x^{n_r}(t^+)\right\} \right)}{\log K^2g_x^{n_r}(t^+)} \to s^*
\]
for $\mu$-a.e.\ $x \in X$.

Hence $\sigma_\mu^-(x,t) \leq -s^*\int \log \lambda\,d\mu/\int \log |T'|\,d\mu$ for
$\mu$-a.e.\ $x\in X$.

The argument for the lower bound on $\sigma_\mu^-(x,t)$ is similar.  Let $t^- < t$.  Then $t^- < t-r$ provided that $r$ is sufficiently small.  We have
\[
 \mu( \{ y \in X \mid u(y) > t^- \}) 
 \geq
 \mu( \{ y \in X \mid u(y) > t-r \})
\]
We have that
\begin{eqnarray*}
 \frac{\mu(\{ y \in B_r(x) \mid u(y)>t^-\})}
 {\mu(B_r(x))}
 & \leq &
 \frac{
 \mu\left( \left\{ y \in \bigcup_{j=1}^M A_{n_r(x_j)}(x_j) \mid u(y) > t^- \right\} \right)}
 {\mu(A_{n_r(x_1)}(x_1))}
 \\
 & = &
 \sum_{j=1}^M\frac{\mu(\{ y \in A_{n_r(x_j)}(x_j) \mid u(y) > t^- \} )}
 {\mu(A_{n_r(x_1)}(x_1))}.
\end{eqnarray*}

An argument similar to that in the proof of Lemma~\ref{lem:kellerbdddist} shows that there exists $D>1$, independent of $r$, such that
\[
 \frac{\mu(\{ y \in A_{n_r(x_j)}(x_j) \mid u(y) > t^- \} )}
 {\mu(A_{n_r(x_1)}(x_1))}
 \leq
 D \mu(T^{n_r(x_j)} \{ y \in A_{n_r(x_j)}(x_j) \mid u(y) > t^-\} ).
\]
We claim that
\[
 T^{n_r(x_j)} \{ y \in A_{n_r(x_j)}(x_j) \mid u(y) > t^-\}
 \subset \{ z \in X \mid u(z) > K^{-1}g^{n_r(x_j)}_x(t^-)\}.
\]
To see this, let $y \in A_{n_r(x_j)}(x_j)$ be such that $u(y) > t^-$.
Let $z=T^{n_r(x_j)}(y)$.  Then, as $g^{n_r(x_j)}_y(\cdot)$ is
orientation preserving and using Lemma~\ref{lem:boundeddistortiong},
we have $u(z) = u(T^{n_r(x_j)}y) = g^{n_r(x_j)}_y(u(y)) >
g^{n_r(x_j)}_y(t^-) \geq K^{-1} g^{n_r(x_j)}_x(t^-)$.

Hence we have
\begin{eqnarray*}
 \frac{\mu(\{ y \in B_r(x) \mid u(y)>t^-\})}
 {\mu(B_r(x))}
 & \leq &
 D\sum_{j=1}^M \mu\left( \left\{ z \in X \mid u(z) >
 K^{-1}g^{n_r(x_j)}_x(t^-) \right\} \right)
\end{eqnarray*} 
so that, by Lemma~\ref{lem:uppertechlemma},
\begin{eqnarray*}
 \sigma_\mu^-(x,t) & \geq &
 \liminf_{r\to 0}
 \frac{1}{\log r}
 \log \left( \frac{ \mu(\{ y \in B_r(x) \mid u(y) > t-r \})}{\mu(B_r(x))} \right) \nonumber \\
 & \leq &
 \limsup_{r \to 0} \frac{1}{\log r}
 \log D \sum_{j=1}^M \mu \left( \left\{ z \in X \mid u(z) > K^{-1}g^{n_r(x_j)}_x(t^-) \right\} \right).
\end{eqnarray*}

Arguing as in the estimates following (\ref{eqn:boundstabindexupper}) we see that for each $j$, $1 \leq j \leq M$ and for $\mu$-a.e.\ $x\in X$,
\[
 \lim_{r\to 0} \frac{1}{\log r} \log \mu\left( \left\{ z \in X \mid u(z) >
 K^{-1}g^{n_r(x_j)}_x(t^-)\right\}\right) = \frac{-s^*\int \log \lambda\,d\mu}{\int
   \log |T'|\,d\mu}.
\]
Hence $\sigma_\mu^-(x,t) \geq  -s^* \int \log \lambda\,d\mu/\int \log |T'|\,d\mu$ for
$\mu$-a.e.\ $x\in X$.
\end{proof}
By Lemma~\ref{lem:onepositiveimpliesotherzero}, we have
$\sigma_\mu^+(x,t)=0$.  This prove Theorem~\ref{thm:stabindex}(i).

\subsection{The lower basin}
\label{subsec:lower}

We now prove Theorem~\ref{thm:stabindex}(ii).  We remark that the
partial hyperbolicity condition (H4) is not needed for this result.
\begin{lemma}
\label{lem:lowerbasin}
Assume that (H1)--(H3) hold.   For $\mu$-a.e.\ $x \in X$ and all $t < u(x)$, there exists $r_0> 0$ such that for all $0 < r < r_0$ we have $\mu \times m (B_r(x,t) \cap \BB^-) = \mu \times m (B_r(x,t))$.
\end{lemma}
\begin{proof}
Suppose $x$ is such that $u(x)$ is defined.  Let $t < u(x)$ and define $\delta_x(t) = u(x)-t>0$.   Choose $n$ such that
\[
 \sum_{j=n}^{\infty} \lambda^{j+1}(x)f(T^jx) \leq \frac{\delta_x(t)}{3}.
\]
As $f, \lambda$ are continuous, we can choose $r' > 0$ such that if $y
\in B_{r'}(x)$, then
\[
 \left| S_{n,\lambda}f(y) -  S_{n,\lambda}f(x) \right| < \frac{ \delta_x(t)}{3}.
\]
Hence for $\mu$-a.e.\ $y \in B_{r'}(x)$ we have
\begin{eqnarray*}
 u(y) & \geq & S_{n,\lambda}f(y) \\
  & \geq &
  S_{n,\lambda}f(x)- \frac{\delta_x(t)}{3} \\
  & = &
  u(x) - \sum_{j=n}^{\infty} \lambda^{j+1}(x)f(T^jx) - \frac{\delta_x(t)}{3}  \\
  & \geq &
  u(x) - 2\frac{\delta_x(t)}{3} =   t + \frac{\delta_x(t)}{3} > t.
\end{eqnarray*}
By Proposition~\ref{prop:graphdefinesbasins}, $(y,t) \in \BB^-$.
Hence, for $r < r_0:= \max \{ r', \delta_x(t)/3\}$ we have that $\mu
\times m (B_r(x,t) \cap \BB^-) = \mu \times m (B_r(x,t))$.
\end{proof}
Hence $\Sigma^{-}_{\mu,r}(x,t)=1$ and $\Sigma^{+}_{\mu,r}(x,t)=0$ provided $r < r_0$.  Hence, by convention, $\sigma_\mu^-(x,t)=0$ and $\sigma_\mu^+(x,t)=\infty$.

\subsection{On the graph}
\label{subsec:ongraph}

We prove Theorem~\ref{thm:stabindex}(iii).  We need the following
version of Lemma~\ref{lem:uppertechlemma}
\begin{lemma}
\label{lem:ongraphtechlemma}
Let $x \in X_u$.  
\begin{itemize}
\item[(i)]
 If $\sigma_\mu^-(x,u(x))$ exists then
\[
 \sigma_\mu^{-}(x,u(x))
 \leq 
 \limsup_{r\to 0} \frac{1}{\log r}
 \log \left( \frac{ \mu( \{ y \in B_r(x) \mid u(y) > u(x) \}) }{\mu(B_r(x))} \right).
\]
\item[(ii)]
 If $\sigma_\mu^+(x,u(x))$ exists then
\[
 \sigma_\mu^{+}(x,u(x))
 \leq 
 \limsup_{r\to 0} \frac{1}{\log r}
 \log \left( \frac{ \mu( \{ y \in B_r(x) \mid u(y) < u(x) \}) }{\mu(B_r(x))} \right).
\]
\end{itemize}
\end{lemma}
\begin{proof}
We prove (i).  First note that
\[
 \{ y \in B_r(x) \mid u(y) > u(x) \} \times [u(x)-r,u(x)] \subset
 B_r(x,u(x)) \cap \BB^-.
\]
Noting that $m([u(x)-r,u(x)]) = r$ and that $\mu \times m(B_r(x,t))=\mu(B_r(x))\times 2r$ we have
\[
 \frac{ \mu(\{ y \in B_r(x) \mid u(y) > u(x) \})}{2\mu(B_r(x))} \leq
 \Sigma_{\mu,r}^-(x,u(x)).
\]
Hence
\[
\frac{1}{\log r}
\log\frac{ \mu(\{ y \in B_r(x) \mid u(y) > u(x) \})}{2\mu(B_r(x))}
\geq
\frac{\log \Sigma_{\mu,r}^-(x,u(x))}{\log r} 
\]
and the result follows by taking the limsup.

The proof of (ii) is analogous, noting that $\{y \in B_r(x) \mid u(y) < u(x) \} \times [ u(x), u(x)+r]  \subset B_r(x,t) \cap \BB^+$.
\end{proof}

The following lemma is a straightforward consequence of Birkhoff's Ergodic Theorem and Proposition~\ref{prop:loynesindex}.
\begin{lemma}
\label{lem:ualongorbitunbdd}
We have that $u(T^nx)$ is unbounded for $\mu$-a.e.\ $x\in X$.
\end{lemma}
\begin{proof}
Let $A_N = \{ x\in X\mid u(x)>N\}$.  By
Proposition~\ref{prop:loynesindex}, for all sufficiently large $N$ we
have $N^{-3s^*/2} < \mu(A_N) < N^{-s^*/2}$; in particular, $\mu(A_N) >
0$ for all $N> N_0$, say.  By Birkhoff's Ergodic Theorem, the set $X_N
:= \{ x\in X \mid T^n(x)\in A_N\ \mbox{for infinitely many}\ n\}$ has
full $\mu$-measure.  Then $\mu(\bigcap_{N=N_0}^{\infty} X_N) =1$ and
consists of points $x$ for which $u(T^nx)$ is unbounded.
\end{proof}

We can now calculate $\sigma_\mu^{-}(x,u(x))$.  
\begin{proposition}
\label{prop:sigminusongraph}
Assume (H1)--(H4) hold.  Then $\sigma_\mu^{-}(x,u(x))=0$ for $\mu$-a.e.\ $x\in X$.
\end{proposition}
\begin{proof}
As $\sigma_\mu^{-}(x,u(x))$ is non-negative, it suffices to show that $\sigma_\mu^{-}(x,u(x)) \leq 0$ for $\mu$-a.e.\ $x\in X$.

By Lemma~\ref{lem:kellerbdddist}, 
\begin{eqnarray*}
 \frac{ \mu(\{ y \in B_r(x) \mid u(y)>u(x) \})}{\mu(B_r(x))}
 & \geq &
 \frac{ \mu(\{ y \in A_{n_r(x_1)}(x_1) \mid u(y)>u(x) \})}
 { \mu\left( \bigcup_{j=1}^M A_{n_r(x_j)}(x_j) \right) } \\
 & \geq &
 D^{-1}
 \mu( T^{n_r(x_1)}\{ y \in A_{n_r(x_1)}(x_1) \mid u(y)>u(x) \}).
\end{eqnarray*}

Let $n_r = n_r(x_1)$ and note that $n_r \to\infty$ as $r\to 0$.  We
claim that
\begin{equation}
\label{eqn:tomsinclusion}
  T^{n_r}\{y \in A_{n_r}(x) \mid u(y) > u(x)\} \supset \left\{ z \in X \mid u(z) > C_{f,\lambda}+C_{\lambda}u(T^{n_r}x)\right\}
\end{equation}
where $C_{f,\lambda},C_{\lambda}$ are as in
(\ref{eqn:upperboundeddist}), (\ref{eqn:bdddistlambda}), respectively.
To see this, let $z$ be such that $u(z) >
C_{f,\lambda}+C_{\lambda}u(T^{n_r}x)$.  As $T|_{A_{n_r}(x)} :
A_{n_r}(x) \to X$ is a bijection, for each $z \in X$ there is a unique
$y \in A_{n_r}(x)$ for which $T^{n_r}y = z$.  Recall that $u(x) =
S_{{n_r},\lambda}f(x) + \lambda^{n_r}(x)u(T^{n_r}x)$.  Hence
\begin{eqnarray*}
 u(y)-u(x) & = &
 S_{n_r,\lambda}f(y) - S_{n_r,\lambda}f(x)
 + 
  \lambda^{n_r}(y)u(T^{n_r}y)-  \lambda^{n_r}(x)u(T^{n_r}x)  \\
 & \geq &
 -C_{f,\lambda}\lambda^{n_r}(y) + \lambda^{n_r}(y)u(T^ny)-  \lambda^{n_r}(x)u(T^{n_r}x)\ \mbox{by}\ (\ref{eqn:upperboundeddist}) \\
 & \geq &
 -C_{f,\lambda}\lambda^{n_r}(y) + \lambda^{n_r}(y)u(T^{n_r}y)- C_{\lambda}\lambda^{n_r}(y)u(T^{n_r}x)\ \mbox{by}\ (\ref{eqn:bdddistlambda}) \\
 & = &
 \lambda^{n_r}(y)\left( -C_{f,\lambda} + u(z) - C_{\lambda}u(T^{n_r}x) \right).
\end{eqnarray*}
As $\lambda > 0$, it follows that $u(y)>u(x)$.  This proves
(\ref{eqn:tomsinclusion}).

Hence
\[
 \frac{ \mu(\{ y \in B_r(x) \mid u(y)>u(x) \})}{\mu(B_r(x))}
 \geq
 D^{-1} 
 \mu\left(\left\{ z \in X \mid u(z) > C_{f,\lambda}+C_{\lambda}u(T^{n_r}x) \right\}\right).
\]
Hence
\begin{eqnarray*}
 \sigma_\mu^-(x,u(x))
  & \leq & 
 \limsup_{r\to 0} \frac{1}{\log r} \log D^{-1} \mu\left(\left\{ z \in
 X \mid u(z) > C_{f,\lambda}+C_{\lambda}u(T^{n_r}x)
 \right\}\right). \nonumber \\
 & = &
  \limsup_{r\to 0} \frac{1}{\log r} \log \mu\left(\left\{ z \in X \mid
  u(z) > C_{f,\lambda}+C_{\lambda}u(T^{n_r}x) \right\}
  \right). \label{eqn:calclowerongraph}
\end{eqnarray*}
We bound and split the right-hand side of (\ref{eqn:calclowerongraph}) as
\begin{equation}
\label{eqn:calclowerongraph1}
 \frac{n_r}{-\log r} \times
 \frac{\log^+ \left(C_{f,\lambda}+C_{\lambda}u(T^{n_r}x)\right)}{n_r} \times
 \left(
 \frac{- \log  \mu\left(\left\{ z \in X \mid u(z) > C_{f,\lambda}+C_{\lambda}u(T^{n_r}x)\right\}\right)}{\log \left(C_{f,\lambda}+C_{\lambda}u(T^{n_r}x)\right)} \right).
\end{equation}

By Lemma~\ref{lem:ualongorbitunbdd}, for $\mu$-a.e.\ $x \in X$ we can
choose a sequence $n'_k \to \infty$ such that
$C_{f,\lambda}+C_{\lambda}u(T^{n'_k}x) \to \infty$.  Hence, by
Proposition~\ref{prop:loynesindex}, the third term in
(\ref{eqn:calclowerongraph1}) converges to $s^*$ as $n'_k \to \infty$.
By Corollary~\ref{cor:logintegrable}, $\log^+ u$, and so $\log^+
(C_{f,\lambda}+C_{\lambda}u)$, is integrable.  It is then a well-known
corollary of Birkhoff's Ergodic Theorem that
\[
 \lim_{n\to\infty} \frac{1}{n} \log^+ \left(C_{f,\lambda}+C_{\lambda}u(T^nx)\right) = 0
\]
for $\mu$-a.e.\ $x \in X$.  Finally, the first term in
(\ref{eqn:calclowerongraph1}) converges $\mu$-a.e.\ to $1/\int \log
|T'|\,d\mu$ by (\ref{eqn:morancover}).

Hence $\sigma_\mu^-(x,u(x)) =0$ for $\mu$-a.e.\ $x \in X$.
\end{proof}

We now calculate $\sigma_\mu^{+}(x,u(x))$.  
\begin{proposition}
\label{prop:sigplusongraph}
Assume (H1)--(H4) hold.  Then $\sigma_\mu^{+}(x,u(x))=0$ for $\mu$-a.e.\ $x\in X$.
\end{proposition}
\begin{proof}
Again, it suffices to show that $\sigma_\mu^{+}(x,u(x)) \leq 0$ $\mu$-a.e..

Analogously to the proof of Proposition~\ref{prop:sigminusongraph} we can bound
\[
  \frac{ \mu(\{ y \in B_r(x) \mid u(y)<u(x) \})}{\mu(B_r(x))}
 \geq
 D^{-1} \mu\left(\left\{ z \in X \mid u(z) < C_{f,\lambda} +
 C_{\lambda}u(T^{n_r}x) \right\}\right)
\]
so that
\[
 \sigma_\mu^+(x,u(x)) \leq 
 \limsup_{r \to 0}  
 \frac{1}{\log r} \log \mu\left(\left\{ z \in X \mid u(z) < C_{f,\lambda} + C_{\lambda}u(T^{n_r}x) \right\}\right).
\]
We split the term inside the limsup as
\begin{equation}
\label{eqn:calcupperongraph1}
 \frac{n_r}{\log r} \times
 \frac{\log^+ \left(C_{f,\lambda}+C_{\lambda}u(T^{n_r}x)\right)}{n_r} \times
 \left(
 \frac{ \log  \mu\left(\left\{ z \in X \mid u(z) < C_{f,\lambda}+C_{\lambda}u(T^{n_r}x)\right\}\right)}{\log \left(C_{f,\lambda}+C_{\lambda}u(T^{n_r}x)\right)} \right).
\end{equation}

As in the proof of Proposition~\ref{prop:sigminusongraph}, the first
two terms of (\ref{eqn:calcupperongraph1}) converge to $-1/\int
\log|T'|\,d\mu$ and $0$, respectively, for $\mu$-a.e.\ $x \in X$.

By Lemma~\ref{lem:ualongorbitunbdd}, for $\mu$-a.e.\ $x \in X$ we can
choose a sequence $n'_k \to \infty$ such that
$M_{n'_{k}}:=C_{f,\lambda}+C_{\lambda}u(T^{n'_k}x) \to \infty$.  By
Proposition~\ref{prop:loynesindex}, provided $M$ is sufficiently
large, we have $\mu(\{ z \in X \mid u(z) >M\}) \geq M^{-3s^*/2}$.
Hence
\begin{eqnarray*}
 \frac{ \log \mu(\{ z \in X \mid u(z) < M_{n'_k}\})}{\log M_{n'_k}}
 &= &
 \frac{ \log \left(1- \mu(\{ z \in X \mid u(z) > M_{n'_k}\}) \right)}{\log M_{n'_k}}\\
 & \leq &
 \frac{ \log \left(1-M_{n'_{k}}^{-3s^*/2}\right)}{\log M_{n'_k}}
\end{eqnarray*}
which converges to $0$ as $n'_k \to \infty$.

Hence $\sigma_\mu^+(x,u(x))=0$.
\end{proof}

\section{The multifractal spectrum of the stability index}
\label{sec:multifractal}

Recall that we define $K_{\mu}(\sigma) = \{ x \in X \mid
\sigma_\mu(x,t) = -\sigma\ \mbox{for all}\ t>u(x)\}$.  From the proof
of Theorem~\ref{thm:stabindex} we see that
\[
 K_{\mu}(\sigma) = \left\{ x \in X \mid \lim_{n\to\infty} \frac{s^*
   S_n \log \lambda(x)}{-S_n\log |T'(x)|} = \sigma \right\}.
\]
Thus the Hausdorff dimension of the level sets of the stability index
can be analysed by invoking multifractal analysis, as described in
\cite[for example]{pesinweiss:97, pesin:97}.  

We first recall multifractal analysis as it is formulated in
\cite{pesinweiss:97}.  Let $\psi$ be \Holder continuous and suppose
that $\log \psi$ is normalised (so that $P(\log \psi)=0$).  Define
$S(q)$ by $P(-S(q)\log |T'| + q\log \psi)=0$ and let $\mu_q$ denote
the equilibrium state with potential $-S(q)\log |T'| + q\log \psi$.
Let $\sigma(q) = -S'(q) = \int \log \psi\,d\mu_q/\int \log |T'|^{-1}\,d\mu_q$.   Suppose that $\log |T'|$ is not cohomologous to $\log \psi$ plus a constant.  Then
$S(q)$ is a strictly convex analytic function and is the Legendre transform
pair of the function $f(\sigma) = \dim_H \{ x \in X \mid
\lim_{n\to\infty} S_n\log\psi(x)/S_n \log |T'(x)|^{-1} = \sigma \}$, so that $f(\sigma(q)) = S(q)+q\sigma(q)$.
Moreover, $f(\sigma(q))$ is defined on the interval
$[\sigma(\infty),\sigma(-\infty)]$.  Finally, $S(q)$ is the
Hentschel--Procaccia dimension spectrum.  We remark that an analysis
of the proofs shows that only the last statement requires $\log
\psi$ to be normalised.  We briefly sketch why $f(\sigma(q))$ is the Legendre transform of $S(q)$.  Note that $\mu_q(K_\mu(\sigma(q)))=1$.  Let $x \in K_{\mu}(\sigma(q))$ then $\prod_{j=0}^{n-1} \psi(T^jx) \sim \prod_{j=0}^{n-1} |T'(T^jx)|^{\sigma}$.  Let $A_n(x)$ is a cylinder of diameter approximately $r$.  Then, by (\ref{eqn:gibbs}), $\mu_q(A_n(x)) \sim \prod_{j=0}^{n-1} |T'(T^jx)|^{-S(q)} \psi(T^jx)^{q} \sim \prod_{j=0}^{n-1} |T'(T^jx)|^{-(S(q)+q\sigma(q))} \sim r^{S(q)+q\sigma(q)}$.  Hence one would expect typical points in $K_{\mu}(\sigma(q))$ to have local dimension $S(q)+q\sigma(q)$.

When considering the multifractal structure of $K_{\mu}(\sigma)$, we recall that we require $\mu_q$ to be such that the invariant graph $u$ is defined $\mu_q$-a.e.  Thus we require $\int \log \lambda\,d\mu_q < 0$.  This places an additional restriction on the set of $q$ for which the multifractal spectrum is defined.  

We prove Proposition~\ref{prop:multifractalmarkov} and Theorem~\ref{thm:multifractal} below.

\subsection{The SRB measure}
\label{subsec:SRBmultifractal}

We first consider the case when $T$ is a uniformly expanding Markov map on $[0,1]$ and $\mu$ is the SRB measure.  In this
case, $\mu$ has potential $\phi = -\log |T'|$.  Note that $P(-\log
|T'|)=0$.  The Loynes exponent $s^*$ is defined by $P(-\log |T'|
+ s^* \log \lambda)=0$.

Define $S(q)$ by $P(-S(q)\log |T'| + s^*q \log \lambda)=0$.   That $S(q)$ is well-defined follows by defining $\phi(q,r) = -r\log|T'|+qs^* \log \lambda$, noting that $\partial
P(\phi(q,r))/\partial r = -\int \log |T'|\,d\nu \not=0$ for an
appropriate measure $\nu$, and using the implicit function theorem.  Standard arguments involving the
analyticity of pressure show that $S(q)$ is analytic.  Note that $S(0)=1$ and $S(1)=1$ (as $s^*$ is the Loynes exponent).  Let $\mu_q$ be the equilibrium state
with potential $-S(q)\log |T'| + s^*q \log \lambda$.

Differentiating
$P(-S(q)\log |T'| + s^*q \log \lambda) = 0$ with respect to $q$ shows
that $S'(q) = -s^* \int \log \lambda\,d\mu_q/\int \log |T'|\,d\mu_q$.
Differentiating $P(-S(q)\log |T'| + s^*q \log \lambda) = 0$ twice with
respect to $q$ and using a standard result from \cite{ruelle:78} shows
that $S''(q) \geq 0$ with equality if and only if $s^*\log \lambda$ and
$S'(q)\log|T'|$ are cohomologous up to a constant.  Recall if two functions $f, g$ are
cohomologous up to a constant $c$ then $\int f\,d\nu=\int g\,d\nu + c$ for any $T$-invariant
measure $\nu$.  Note that, by (H3), $\int s^* \log \lambda\,d\zeta >
0$ and $\int s^* \log \lambda\,d\mu < 0$.  However, as $\log |T'| \geq
0$, we have $\int S'(q)\log |T'|\,d\zeta$ and $\int S'(q)\log
|T'|\,d\mu$ have the same sign (or are zero, if $S'(q)=0$).  Hence
$S''(q)>0$ and so $S(q)$ is a strictly convex function.

As $S(q)$ is strictly convex, $S(0)=1$ and $S(1)=1$, there exists a
unique $q^* \in (0,1)$ such that $S'(q^*)=0$.  If $q<q^*$ then $\int \log \lambda\,d\mu_q<0$.  Hence $\mu_q(X_u) =1$ so that $u$ is defined $\mu_q$-a.e.

Let $\sigma(q) = -S'(q) = -s^* \int \log\lambda\,d\mu_q/\int \log |T'|\,d\mu_q$.    Then standard arguments 
from \cite{pesinweiss:97} (sketched above) show that $\dim_H K_{\mu}(\sigma(q)) = S(q) +q\sigma(q)$, the Legendre transform of $S(q)$, and that this is defined for $q \in (-\infty,q^*)$.  This is illustrated in Figure~\ref{fig:multifrac} below.   

When $q=0$ we have that $\mu_q=\mu$, the SRB measure.  Hence
\[
 {\dim}_H \left\{ x \in X \mid \sigma_\mu(x,t) = \frac{s^* \int \log\lambda\,d\mu}{\int \log |T'|\,d\mu}\  \mbox{for all}\ t > u(x) \right\} = 1.
\]

\begin{figure}[h]
\begin{center}
\includegraphics[width=90mm]{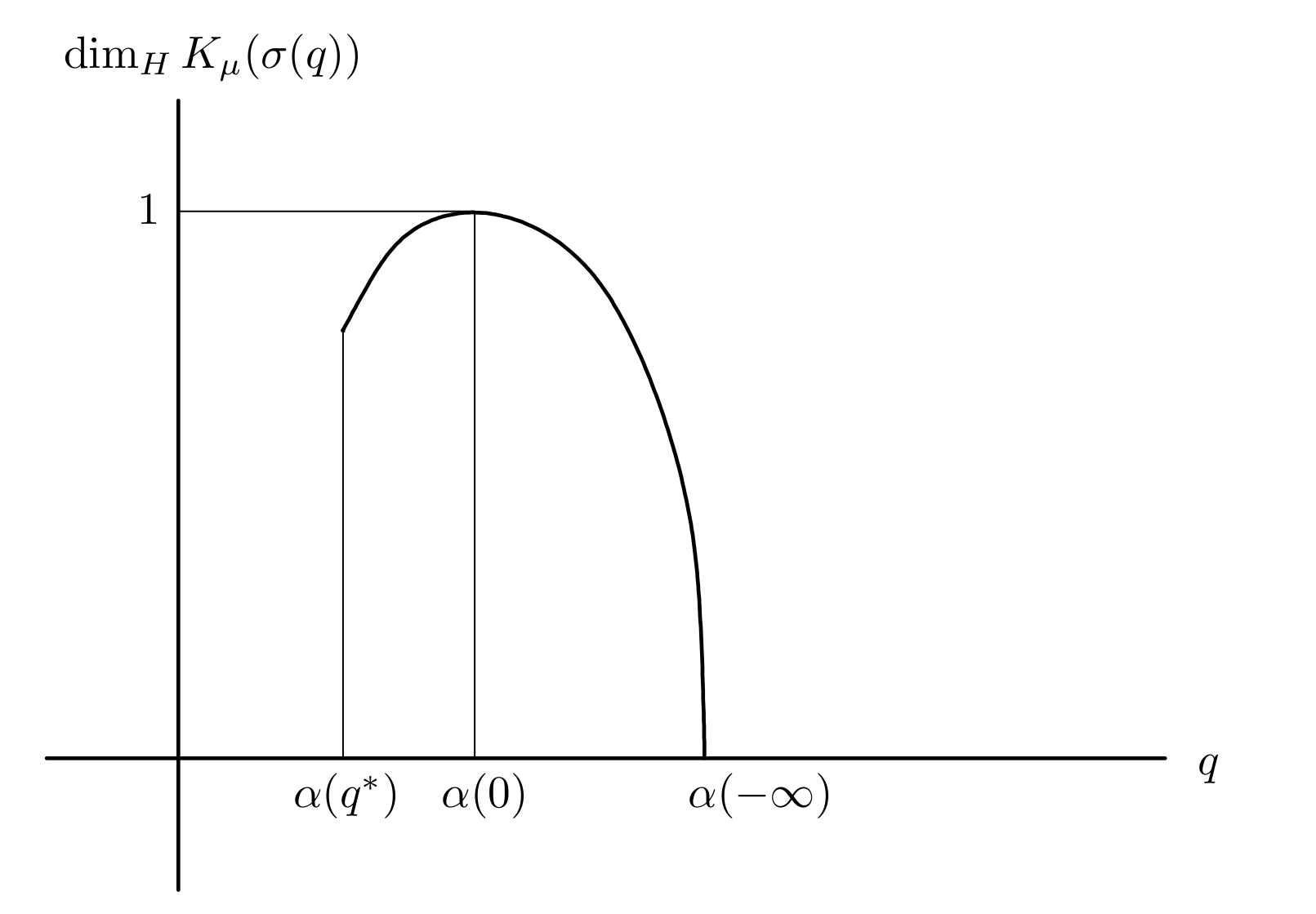}
\end{center}
\caption{The multifractal spectrum of the stability index in the case of the SRB measure.}
\label{fig:multifrac}
\end{figure}

\subsection{The general case}

Assume that (H1)--(H4) hold and that $\mu$ is an arbitrarily
equilibrium state corresponding to a \Holder potential.  As above,
define $S(q)$ by $P(-S(q)\log |T'| + qs^* \log\lambda)=0$ and let
$\mu_q$ denote the equilibrium state with potential $-S(q)\log |T'| +
qs^* \log\lambda$.  As in \S\ref{subsec:SRBmultifractal}, $S(q)$ is
well-defined, strictly convex and $S'(q) = s^* \int \log
\lambda\,d\mu_q/\int \log |T'|\,d\mu_q$.  Noting that $S(0)$ solves
the pressure equation $P(-S(0)\log |T'|)$ we see that $S(0)= \dim_H
X$.

We first show that there exists $q \in \RR$ such that $\int \log\lambda\,d\mu_q < 0$.  We remark that we do not necessarily have that $q\geq 0$.

The following is proved in \cite{simpelaere:94} (we note that
\cite{simpelaere:94} assumes $\log \lambda$ to be such that $P(\log
\lambda)=0$, however the proof can be easily modified to hold without
this assumption).
\begin{lemma}[\cite{simpelaere:94}]
\label{lem:simpelaere}
Assume that (H1)--(H3) hold.  Let $S(q)$ be defined as above.  Then
\[
 S(q) =  \inf \left\{
  \frac{ h(\nu) + qs^*\int \log \lambda\,d\nu}{\int \log |T'|\,d\nu} \mid \nu\ \mbox{is a}\ T\mbox{-invariant probability measure} \right\}
\]
where $h(\nu)$ denotes the entropy of $T$ with respect to $\mu$.
\end{lemma}
Lemma~\ref{lem:simpelaere} implies the following result (cf.~\cite{schmeling:99}).
\begin{lemma}
Assume (H1)--(H3) hold.  Then there exists $q \in \RR$ such that $\int
\log \lambda\,d\mu_q < 0$.
\end{lemma}
\begin{proof}
Suppose for a contradiction that $\int \log \lambda\,d\mu_q \geq 0$
for all $q \in \RR$.  Then $S'(q) \geq 0$ for all $q \in \RR$.  As $S$
is strictly convex, it follows that $\alpha_0 := \inf_{q\in\RR} S'(q)
= \lim_{q\to-\infty} S'(q) \geq 0$.  We show that this cannot happen.
Recall that if $m$ is any $T$-invariant probability measure then
$h_m(T) \leq h_{\mathrm{top}}(T)$, the topological entropy of $T$.  By
(\ref{eqn:varprinciple}), $S(q) = (h(\mu_q) +qs^* \int \log
\lambda\,d\mu_q)/\int \log |T'|\,d\mu_q$.

Let $\veps > 0$.  Choose $q<0$ such that $\alpha_0 < S'(q) < \alpha_0+\veps$. Then
\begin{eqnarray*}
q\alpha_0 & > &
\frac{qs^* \int \log \lambda\,d\mu_q}{\int \log |T'|\,d\mu_q} - q\veps \\
& > &
S(q) - \frac{h(\mu_q)}{\int \log |T'|\,d\mu_q} - q\veps \\
& \geq &
\inf \left\{\frac{ h(\nu) + qs^*\int \log \lambda\,d\nu}{\int \log
  |T'|\,d\nu} \right\} - \frac{h_{\mathrm{top}}(T)}{\log \|T'\|} - q\veps \\
& \geq &
\inf \left\{ \frac{qs^*\int \log \lambda\,d\nu}{\int \log |T'|\,d\nu} \right\}-
\frac{h_{\mathrm{top}}(T)}{\log \|T'\|} - q\veps
\end{eqnarray*}
where both infima are taken over all $T$-invariant probability
measures $\nu$.  Dividing by $q$, letting $q \to -\infty$ and noting
that $\veps>0$ is arbitrary, we have that $\alpha_0 \leq \inf \{ s^*
\int \log \lambda\,d\nu/\int \log |T'|\,d\nu\}$ where the infimum is taken
over all $T$-invariant probability measures.

Taking $\nu=\mu$, by (H3) we see that $\alpha_0 < 0$.  Hence there exists
$\mu_q$ such that $S'(q) < 0$, a contradiction.
\end{proof}
Repeating the above argument with $q>0$ and letting $q \to \infty$
shows that $\sup_{q\in\RR} S'(q)= \lim_{q\to\infty} S'(q) \geq \sup\{ s^*
  \int\log \lambda\,d\nu/\int \log |T'|\,d\nu\}$ where the last
  supremum is taken over all $T$-invariant probability measures.
  Taking $\nu=\zeta$, we see that $S'(q) > 0$ for all sufficiently
  large $q$.

As $S(q)$ is strictly convex, we have $S'(q)$ is increasing.  Hence
there exists a unique $q^* \in \RR$ such that $S'(q^*)=0$.  Note that
if $q<q^*$ then $\int \log \lambda\,d\mu_q < 0$; hence, for $q<q^*$,
we have that $\mu_q(X_u)=1$ so that the invariant graph $u$ is defined
$\mu_q$-a.e.  Let $\sigma(q) = -s^* \int \log \lambda\,d\mu_q/\int
\log |T'|\,d\mu_q$.  Then standard arguments from \cite{pesinweiss:97}
(and sketched above) show that $\dim_H K_{\mu}(\sigma(q))=
S(q)+q\sigma(q)$, the Legendre transform of $S(q)$, and that this is
defined for $q \in (-\infty,q^*)$.  The two cases are illustrated in
Figure~\ref{fig:multifractalgeneral}.

The Hausdorff dimension of $\{ x \in X \mid \sigma_\mu(x,t) = s^* \int
\log\lambda\,d\mu/\int \log |T'|\,d\mu\ \mbox{for all}\ t > u(x) \}$
is given by the unique $q \in (-\infty, q^*)$ for which $S'(q) = s^*
\int \log \lambda\,d\mu/\int \log |T'|\,d\mu$.  (In general, unless
$\mu = \mu_q$ for some $q$, i.e.\ $\phi$ is cohomologous to $-S(q)\log
|T'|+qs^*\log\lambda$, then we cannot expect to find a closed form for
$q$.)

\begin{figure}[h]
\begin{center}
\includegraphics[width=120mm]{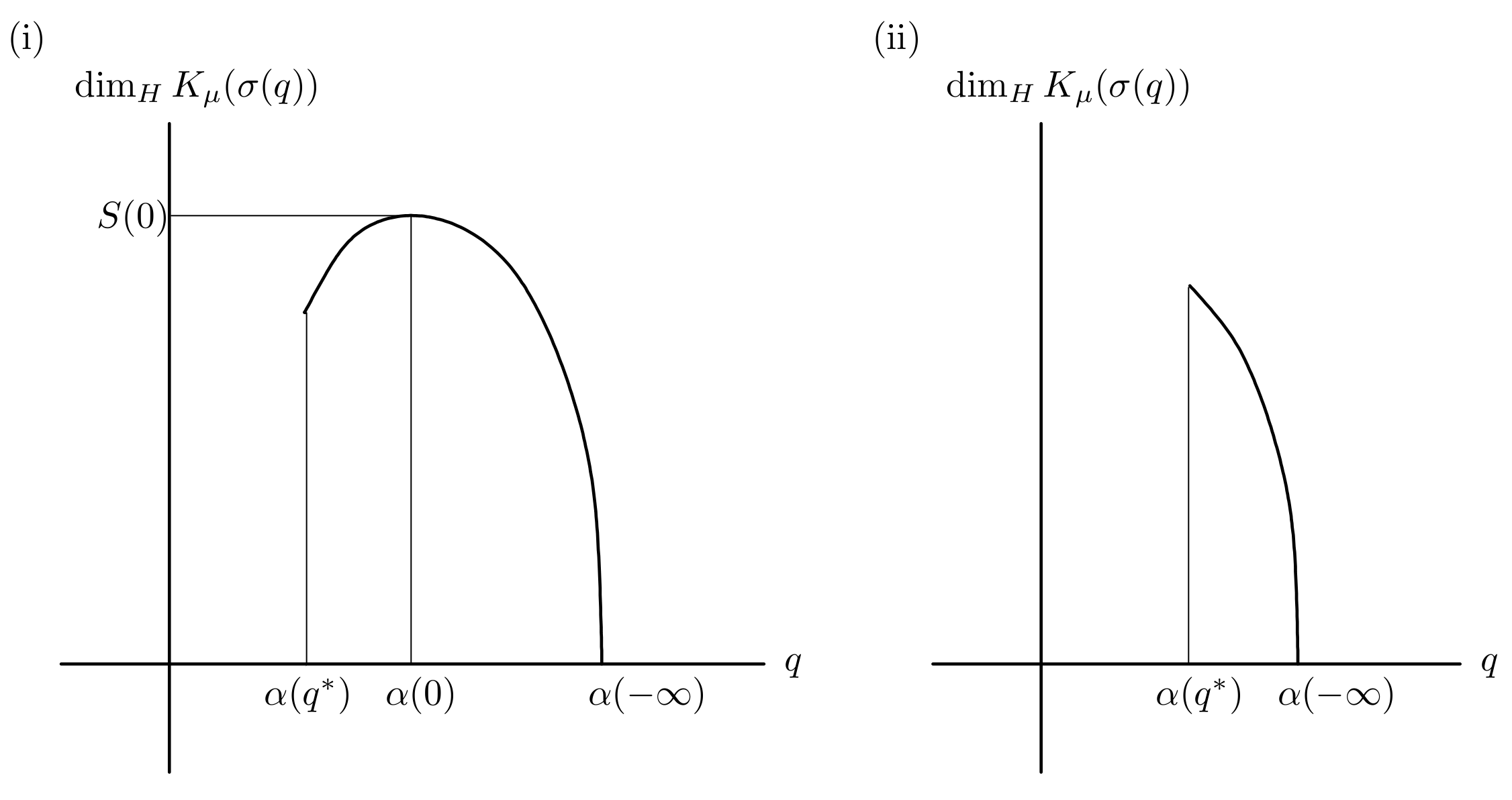}
\end{center}
\caption{The multifractal spectrum of the stability index in the general case when (i) $q^* >0$, (ii) $q^*<0$.}
\label{fig:multifractalgeneral}
\end{figure}

\begin{flushleft}
 Charles Walkden, School of Mathematics, The University of Manchester,
 Oxford Road, Manchester M13 9PL, U.K.,
 email: \texttt{charles.walkden@manchester.ac.uk}
\end{flushleft}

\begin{flushleft}
 Tom Withers, School of Mathematics, The University of Manchester,
 Oxford Road, Manchester M13 9PL, U.K.,
 email: \texttt{tomwithers1991@hotmail.co.uk}
\end{flushleft}

\end{document}